\documentclass[11pt]{article}
\pdfoutput=1
\usepackage{amsmath,amsfonts,amssymb,amsthm,mathrsfs}
\usepackage[a4paper,vmargin={3.5cm,3.5cm},hmargin={2.5cm,2.5cm}]{geometry}
\usepackage[font=sf, labelfont={sf,bf}, margin=1cm]{caption}
\usepackage{graphicx,graphics}
\usepackage{epsfig}
\usepackage{latexsym}
\usepackage[applemac]{inputenc}
\usepackage{ae,aecompl}
\usepackage[english]{babel}
\usepackage[colorlinks=true]{hyperref}
\usepackage{enumerate}
\usepackage{bbm}

\date{}

\newcommand{\Z}{\mathbb{Z}}
\newcommand{\expec}{\mathbb{E}}
\newcommand{\prob}{\mathbb{P}}

\newcommand{\ve}{\varepsilon}
\newcommand{\wt}{\widetilde}
\newcommand{\E}{\mathbb{E}}

\newcommand{\rmd}{\mathrm{d}}

\newcommand{\qseq}{\mathbf{q}}
\newcommand{\map}{\mathfrak{m}}

\newcommand{\rootface}{f_{\mathrm{r}}}

\newcommand{\maps}{\mathcal{M}}

\newcommand{\faces}{\mathsf{Faces}}

\newcommand{\half}{{\textstyle\frac12}}

\newtheorem{theorem}{Theorem}
\newtheorem{theorema}{Theorem}

\newtheorem{propositiona}[theorema]{Proposition}
\newtheorem{remark}[theorem]{Remark}

\newtheorem{proposition}[theorem]{Proposition}
\newtheorem{lemma}[theorem]{Lemma}
\newtheorem{corollary}[theorem]{Corollary}
\newtheorem*{open*}{Open question}
\numberwithin{theorem}{section}

\def\llbracket{[\hspace{-.10em} [ }

\title{\bf \textsc{Geometry of infinite planar maps with high degrees}}

\author{ \text{Timothy Budd}\thanks{NBI, University of Copenhagen \& IPhT, CEA, Universit\'e Paris--Saclay, E-mail: timothy.budd@cea.fr} \ \ \&  \text{Nicolas Curien}\thanks{Universit\'e Paris--Sud, Universit\'e Paris--Saclay, E-mail: nicolas.curien@gmail.com}}

\begin{document}
	\maketitle

	\abstract{
		We study the geometry of infinite random Boltzmann planar maps with vertices of high degree.  These correspond to the duals of the Boltzmann maps associated to a critical weight sequence $(q_{k})_{ k \geq 0}$ for the faces with polynomial decay $k^{-a}$ with $a \in ( \frac{3}{2}, \frac{5}{2})$ which have been studied by Le Gall \& Miermont as well as by Borot, Bouttier \& Guitter. We show the existence of a phase transition for the geometry of these maps at $a = 2$. In the dilute phase corresponding to $a \in (2, \frac{5}{2})$ we prove that the volume of the ball of radius $r$ (for the graph distance) is of order $r^{ \mathsf{d}}$ with $$  \mathsf{d}= \frac{a-  \half}{a-2},$$ and we provide distributional scaling limits for the volume and perimeter process. In the dense phase corresponding to $ a \in ( \frac{3}{2},2)$ the volume of the ball of radius $r$ is exponential in $r$. We also study the first-passage percolation (\textsc{fpp}) distance with exponential edge weights and show in particular that in the dense phase the \textsc{fpp} distance between the origin and $\infty$ is finite.
		The latter implies in addition that the random lattices in the dense phase are transient. The proofs rely on the recent peeling process introduced in \cite{Bud15} and use ideas of \cite{CLGpeeling} in the dilute phase.}
	\begin{figure}[!h]
		\begin{center}
			\includegraphics[width=8.0cm]{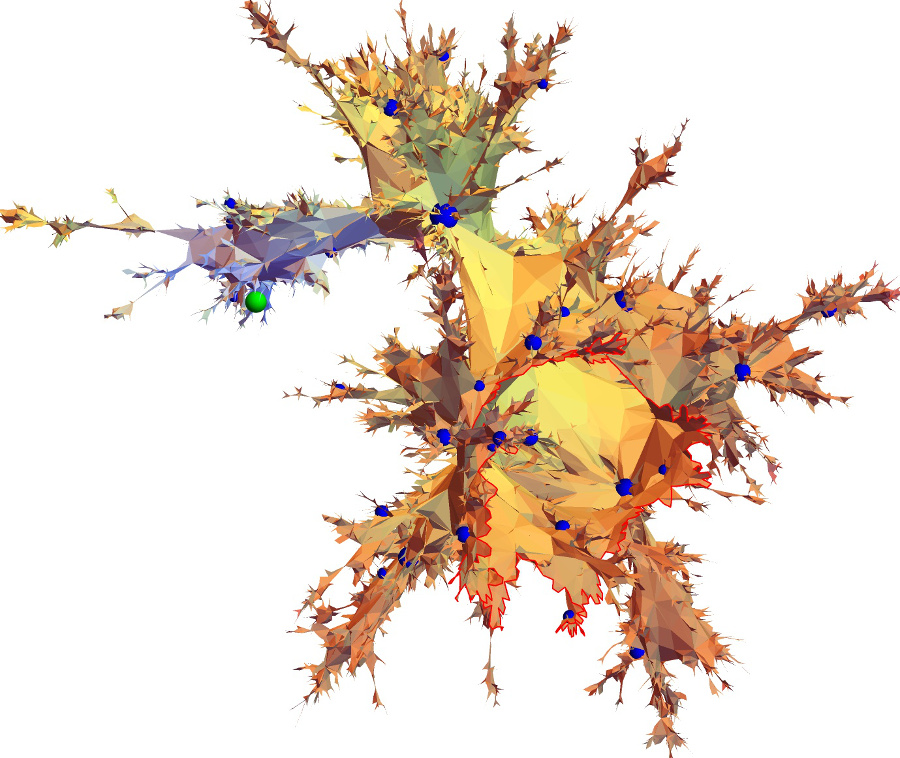}\includegraphics[width=7.8cm]{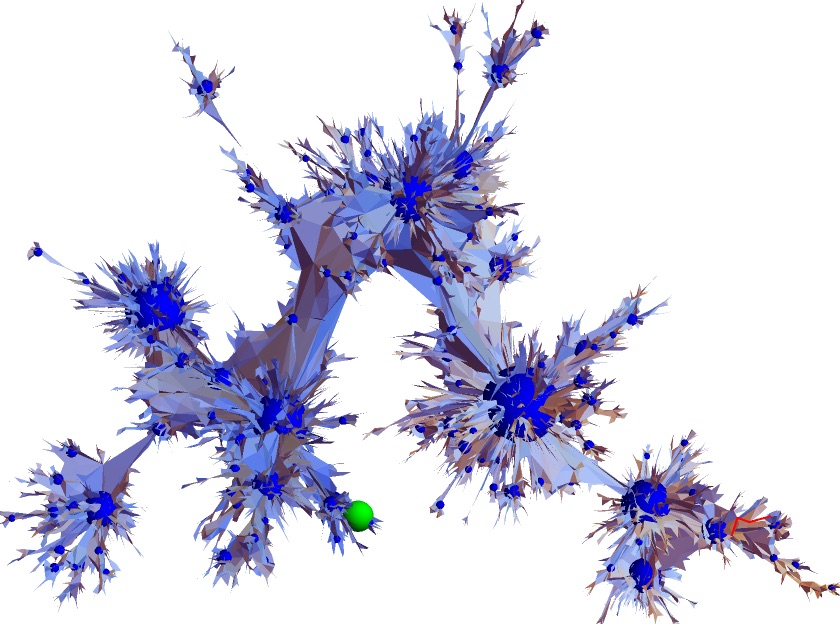}
			\caption{Two representations of the neighborhood of the root in infinite Boltzmann maps with large degree vertices in the dilute case (left) and dense case (right). The root is represented by a green ball, while the high degree vertices are represented by blue balls of size proportional to the degree. The boundary is colored red.}
		\end{center}
		\vspace{-0.6cm}
		
	\end{figure}

\section{Introduction}
Whereas the geometry of random planar maps (\textsc{rpm}) converging towards the Brownian map is by now pretty well understood, the problem remains open for many other models of \textsc{rpm}. Famous examples of these are the \textsc{rpm} coupled with an $O(n)$ model, $n \in (0,2)$, where information about distances remains out of reach. In \cite{LGM09} Le Gall and Miermont studied the geometry of \textsc{rpm} with large faces which correspond to the gaskets of the above planar maps coupled with an $O(n)$ model and in particular introduced their (conjectural) scaling limits. In this work we study the geometry of the dual of these maps which yields new interesting geometric phenomena. \medskip

\paragraph{Infinite Boltzmann planar maps.} Let us first recall the model of planar maps we are dealing with. As usual, all planar maps in this work are rooted, i.e.~equipped with a distinguished oriented edge; also for technical simplicity we will  only consider \emph{bipartite} planar maps (all faces have even degree). We denote by $ \mathcal{M}_{n}$ the set of all finite bipartite planar maps with $n$ vertices. Given a non-zero sequence $ \mathbf{q}= (q_{k})_{ k \geq 1}$ of non-negative numbers we define a measure $w$ on the set of all bipartite planar maps by the formula
\begin{equation} \label{eq:defface}
w(\map) := \prod_{f\in\faces(\map)} q_{\deg(f)/2},
\end{equation}
for every $ \map \in \cup_{n \geq 0} \mathcal{M}_{n}$. We shall assume that $w$ is admissible, meaning that $w$ is a finite measure on $\cup_{n \geq1} \mathcal{M}_{n}$. We shall also suppose that $ \mathbf{q}$ is critical in the sense of \cite[Equation (3)]{MM07}  (see \cite{Bud15}, recalled in Proposition \ref{prop:qnubijection} below, for an equivalent definition). For $n \geq 0$, provided  that $w( \mathcal{M}_{n}) \ne 0$, we define a random planar map $B_{n}$ called the $ \mathbf{q}$-Boltzmann random map with $n$ vertices whose law is $w( \cdot \mid  \cdot \in \mathcal{M}_{n})$.  Under these conditions we have the following convergence in distribution for the local topology along the integers $n$ for which $w( \mathcal{M}_{n}) \ne 0$
\[ B_{n} \xrightarrow[n\to\infty]{(d)} B_{\infty},\] where $B_{\infty}$ is an infinite random rooted bipartite planar map with only one end, which is called the infinite $ \mathbf{q}$-Boltzmann planar map \cite{BS13,St14}. As in \cite[Section 2.2]{LGM09} or in \cite{BBG11}, we focus henceforth on the case when the critical and admissible weight sequence $ \mathbf{q}$ is non-generic, in particular satisfies for some $c, \kappa>0$
 \begin{eqnarray} \label{eq:asymptoticqk} q_{k} \sim c\, \kappa^{k-1} \, k^{-a} \quad \mbox{ as } k \to \infty, \qquad \mbox{ for }  a\in \left( \frac{3}{2}, \frac{5}{2}\right).  \end{eqnarray} The reader should keep in mind that the values of $c, \kappa$ and $( q_{k})_{k \geq 1}$ need to be fine-tuned in order to have the desired criticality property, see the above references and Section \ref{sec:enume} for details (alternatively the material reader may also use the concrete sequences given in Section \ref{sec:particular}). For this choice of $ \mathbf{q}$ the random Boltzmann maps $B_{n}$ possess ``large faces'' and their scaling limits (at least along subsequences) are given by the stable maps of Le Gall and Miermont \cite{LGM09} (this is a family of random compact metric spaces that  look like randomized versions of the Sierpinski carpet or gasket).  Our main object of study here\footnote{We have decided to introduce our main character as the dual map of $B_{\infty}$ rather than starting with a Boltzmann measure similar to \eqref{eq:defface} but with weights on the vertices. We hope that this will help the reader navigate through the needed references \cite{LGM09,BBG11,Bud15} which deal with weights on the faces.} is the dual map $ B_{\infty}^\dagger$ of $B_{\infty}$ whose vertices are the faces of $B_{ \infty}$ and edges are dual to those of $B_{\infty}$. The origin (or root vertex) of $ B_{ \infty}^\dagger$ is the root face $ \rootface$ of $B_{\infty}$ lying on the right of its root edge, while the root edge of $B_{ \infty}^\dagger$ is taken to be the unique edge starting at the origin and intersecting the root edge of $B_{\infty}$. The large faces of $ B_{\infty}$ turn into large degree vertices in $B_{\infty}^\dagger$ and our goal is to understand the effect of this change on the large scale metric structure. For $r \geq 0$, we denote by $ \mathrm{Ball}_{r}^\dagger( B_{\infty})$ the submap of $ B_{\infty}$ obtained by keeping the faces which are at dual distance at most $r$ from the root face of $B_{\infty}$ and consider its hull 
 \[  \overline{\mathrm{Ball}}^\dagger_{r}( B_{\infty})\]
 made by adding to $ \mathrm{Ball}^\dagger_{r}( B_{\infty})$ all the finite connected components of its complement in $ B_{\infty}$ (recall that $ B_{\infty}$ is one-ended). Our main results describe the evolution of the volume and (a version of) the perimeter  of $\overline{\mathrm{Ball}}^\dagger_{r}( B_{\infty})$ as $r$ varies. 
 \paragraph{Results.} When $a \in (2; \frac{5}{2})$ --the so-called \emph{dilute phase}-- we show (Theorem \ref{thm:scalinglayers}) that the volume of the ball of radius $r$ in $ B_{\infty}^\dagger$ (e.g. measure in terms of the number of faces, i.e.~vertices of $ B_{\infty}$) is polynomial in $r$
 \begin{eqnarray} \label{eq:resultdilute} \mathsf{Volume}\big( \overline{\mathrm{Ball}}_{r}^\dagger( B_{\infty})\big) \approx r^{  \mathsf{dim}_{a}} \qquad \mbox{ where }\quad  \mathsf{dim}_{a} = \frac{a-\half}{a-2} \in (4,\infty).  \end{eqnarray} 
 The exponent $ \mathsf{dim}_{a}$ is called the volume growth exponent or sometimes in physics literature the ``Hausdorff dimension'' of $ B_{\infty}$ since it should correspond to the true Hausdorff dimension of a scaling limit of $ B_{\infty}$ (see below). We also show that $\mathsf{Perimeter}\big( \overline{\mathrm{Ball}}_{r}^\dagger( B_{\infty})\big) \approx r^{1/(a-2)}$ and in fact we obtain the limit in distribution of the rescaled volume and perimeter processes in the same spirit as the results of \cite{CLGpeeling}, see Theorem \ref{thm:scalinglayers}. The value of $\mathsf{dim}_{a}$ should be contrasted to the case of Infinite Boltzmann maps with faces of bounded degree, where the volume growth exponent equals 4, a value which is only approached when $a\to 5/2$ (see also our discussion below).
 
 The above exponents explode when $a\downarrow 2$ indicating a phase transition at this value. This is indeed the case and we prove (Theorem \ref{thm:dense}) that when $a \in ( \frac{3}{2}; 2)$ --the so-called \emph{dense phase}-- the volume and the perimeter of the ball of radius $r$ in $B^\dagger_{\infty}$ grow exponentially with $r$ 
\begin{eqnarray}
\mathsf{Perimeter}\big( \overline{\mathrm{Ball}}_{r}^\dagger( B_{\infty})\big) \approx e^{r \mathsf{c}_{a}}\quad \mbox{ and } \quad 
\mathsf{Volume}\big( \overline{\mathrm{Ball}}_{r}^\dagger( B_{\infty})\big) \approx e^{r (a- \half)\mathsf{c}_{a}}
\label{eq:resultdense}
\end{eqnarray}
for some constant $ \mathsf{c}_{a}>0$  which is expressed in terms of a certain Lévy process of stability index $a-1~\in~( \frac{1}{2};1)$. In the above results the perimeter is computed in terms of number of edges and not in terms of number of vertices (see Section \ref{sec:algodual} for the precise definition). Although this distinction is irrelevant in \eqref{eq:resultdilute}, we show that it is crucial in the dense phase since we prove that $B_{\infty}^\dagger$ has infinitely many cut vertices separating the origin from infinity. Our results show that the geometry of $B_{\infty}^\dagger$ is much different from the geometry of $ B_{\infty}$ (for their respective graph distances). Indeed, extrapolating the work of Le Gall and Miermont \cite{LGM09} one should get that for $a\in(\frac{3}{2};\frac{5}{2})$ the volume of (hulls of) balls in $ B_{\infty}$ should scale as 
\[ \mathsf{Volume}\left( \overline{ \mathrm{Ball}}_{r}( B_{ \infty})\right) \approx r^{2a-1}.\]
Comparing the last display to \eqref{eq:resultdilute} and \eqref{eq:resultdense} we see that the distances in the dual map $ B_{\infty}^\dagger$ are deeply modified. This might be unsurprising since when passing to the dual, the large degree faces become large degree vertices which act as ``hubs'' and shorten a lot the distances. This contrasts with the case of ``generic'' random maps (e.g.~uniform triangulations or quadrangulations) where the primal and dual graph distances are believed to be the same at large scales up to a constant multiplicative factor. This has recently been verified in the case of triangulations \cite{AB14,CLGmodif}.

 We also show similar results when we consider a first-passage percolation (\textsc{fpp}) distance on $B^\dagger_{ \infty}$ instead of the graph distance. Specifically, the edges of $ B_{\infty}^\dagger$ are equipped with independent exponential weights of parameter $1$. These weights are interpreted as random lenghts for the edges and give rise to the associate \textsc{fpp}-distance $ \mathrm{d_{fpp}}$ (this precise model of \textsc{fpp} is the Eden model on $ B_{\infty}$). The result \eqref{eq:resultdilute} still holds in the dilute phase for this distance, with identical scaling limit up to a constant multiplicative factor (see Proposition \ref{prop:scalingeden}). In the dense phase a striking phenomenon occurs: the minimal \textsc{fpp}-length of an infinite path started at the origin $\rootface$ of $B_{\infty}^\dagger$ is finite and moreover its expectation is obtained as the expected number of visits to $0$ of a certain one-dimensional transient random walk (see  Proposition \ref{prop:finiteexpect}). As a corollary we obtain that when $a <2$ the simple random walk on $B_{\infty}^\dagger$ is almost surely transient (Corollary \ref{cor:transient}).

The reader may naturally wonder about the status of the above results in the critical case $a=2$: this will be the content of a companion paper.
\paragraph{Discussion.} In order to discuss our results and explain the terminology of dense and dilute phases, let us  briefly recall some results for the $O(n)$ model on random quadrangulations proved in \cite{BBG11}. A loop-decorated quadrangulation $ (\mathfrak{q},  \mathbf{l})$ is a planar map whose faces are all quadrangles on which non-crossing loops $ \mathbf{l}=(l_{i})_{i \geq 1}$ are drawn (see Fig. 2 in \cite{BBG11}). For simplicity we consider the so-called rigid model when loops can only cross quadrangles through opposite sides.  We define a measure on such configurations by putting 
\[ W_{h,g,n}( ( \mathfrak{q}, \mathbf{l})) = g^{| \mathfrak{q}|}h^{| \mathbf{l}|} n^{\# \mathbf{l}},\]
for $g,h>0$ and $n \in (0,2)$ where $|q|$ is the number of faces of the quadrangulation, $| \mathbf{l}|$ is the total length of the loops and $\# \mathbf{l}$ is the number of loops. Provided that the measure $W_{h,g,n}$ has finite total mass one can use it to define random loop-decorated quadrangulations with a fixed number of vertices. Fix $n \in (0,2)$. For most of the parameters $(g,h)$ these random planar maps are sub-critical (believed to be tree like when large) or generic critical (believed to converge to the Brownian map). However, there exists a critical line with an end point in the $(g,h)$-plane (whose location depends on $n$) at which these planar maps may have different behaviours. More precisely, their gaskets, obtained by pruning off the interiors of the outer-most loops (see Fig. 4 in \cite{BBG11}) are precisely non-generic critical Boltzmann planar maps in the sense of \eqref{eq:asymptoticqk} where 
\[ a = 2 \pm \frac{1}{\pi} \arccos(n/2).\]
The case $a = 2 - \frac{1}{\pi} \arccos(n/2) \in ( \frac{3}{2};2)$ (which occurs when away from the end point) is called the dense phase because the loops on the gasket are believed in the scaling limit to touch themselves and each other.\footnote{The dense phase of the $O(n)$ loop model resembles a critical Fortuin--Kasteleyn (FK) cluster model with parameter $q = n^2$. It is thus conceivable that a suitable notion of the gasket of an $q$-FK model with $q\in(0,4)$ on a random planar map gives rise to a Boltzmann planar map with parameter $a=2 - \frac{1}{\pi} \arccos(\sqrt{q}/2)$. No such correspondence is expected in the dilute phase. } The case $a = 2 + \frac{1}{\pi} \arccos(n/2) \in (2; \frac{5}{2})$ (which occurs exactly at the end point) is called the dilute phase because the loops on the gasket are believed to be simple in the scaling limit and avoiding each other. This heuristic sheds some light on our results: in the dense and dilute phases the appearance of large degree vertices, when passing to the dual of $B_{\infty}$, shortens the distance significantly; this effect obeys a phase transition at $a=2$, because in the dense phase the connections between the large degree vertices are so numerous that the volume growth becomes exponential instead of polynomial.

\begin{figure}[!h]
 \begin{center}
 \includegraphics[width=10cm]{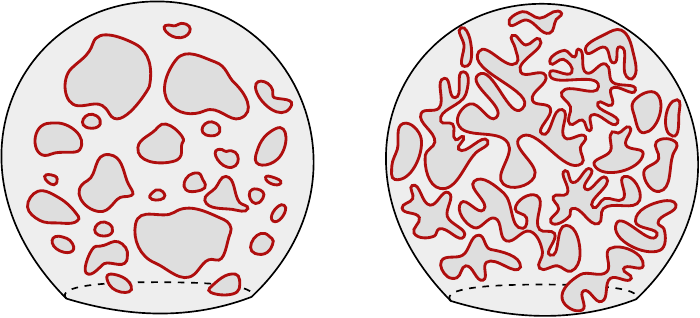}
 \caption{A schematic illustration of $ \mathbf{q}$-Boltzmann \textsc{RPM} in the dilute (left) and dense phase (right). }
 \end{center}
 \end{figure}

\paragraph{Techniques.} Our approach is to explore the map $ B_{\infty}^\dagger$ using the ``lazy'' peeling process recently introduced in \cite{Bud15}. The peeling process was first studied in physics by Watabiki \cite{Wat95} and was the basis for the first derivation of the so-called two-point function \cite{AW95,ADJ97}. It is a stochastic growth process which uses the spatial Markov property of the underlying lattice in order to discover it step by step. A rigorous version of the peeling
process and its Markovian properties was given by Angel \cite{Ang03} in the case of the Uniform Infinite Planar Triangulation (UIPT) and has been one of the key tools to study random triangulations and quadrangulations since then \cite{AB14,Ang03,ACpercopeel,BCKscalingpeeling,CLGpeeling,MN13,BCsubdiffusive,CurKPZ,AR13,CurPSHIT}. The peeling process used in the last references consists roughly speaking in discovering one face at a time. It is well designed to study planar maps with a degree constraint on the faces (such as triangulations or quadrangulations). The peeling process we consider here and which was recently introduced in \cite{Bud15} is different: it discovers one \emph{edge} at a time. The advantage of this ``edge-peeling'' process over the ``face-peeling'' process is that it can be treated in a unified fashion for all models of Boltzmann planar maps. The results we obtain in the dilute case roughly follow from adapting and sharpening the techniques and proofs of \cite{CLGpeeling}. The dense case on the contrary requires a totally new treatment.

\paragraph{Towards a stable sphere.} In a forthcoming work \cite{BBCK16} the authors together with Jean Bertoin and Igor Kortchemski will explore the links between the random maps considered in this work and growth-fragmentation processes. This extends the work \cite{BCKscalingpeeling} where a certain scaling limit of random triangulations was described in terms of a growth-fragmentation process related to the spectrally negative $3/2$-stable process. The new growth-fragmentations involved  may have positive jumps and are related to $\alpha$-stable Lévy processes where $\alpha=a-1$ and with positivity parameter $\rho$ satisfying 
\[ \alpha (1- \rho) = \frac{1}{2}.\]
In the dilute phase $a \in (2; 5/2)$, we conjecture that the random metric spaces $ n^{-1/ \mathsf{dim}_{a}} \cdot B_{n}$ admit a scaling limit  (which we call stable spheres by lack of imagination) which can be constructed from the above growth-fragmentations processes. We expect that these random metric spaces are homeomorphic to the sphere and have Hausdorff dimension $ \mathsf{dim}_{a} = \frac{a-1/2}{a-2}$. A key difference with the Brownian map (corresponding to the case $a= \frac{5}{2}$) is the presence of certain points, ``hubs'', in the metric spaces where a lot of geodesics merge (these correspond to the high degree vertices in the discrete setting). These questions will be addressed in our forthcoming works.  \medskip

We end the discussion with a question that is left open\footnote{Notice that the powerful result of \cite{GGN12} does not apply because the root vertex distribution in $B_{\infty}^\dagger$ has a polynomial tail (and indeed in the dense case those lattices are transient by Corollary \ref{cor:transient}).} by our work:
\begin{open*} Are the random lattices $B_{\infty}^\dagger$ transient or recurrent in the dilute case $a \in (2, 5/2)$?
\end{open*}

\textbf{Acknowledgments: } We are grateful to Igor Kortchemski for comments on a preliminary version of this work and for providing us with reference \cite{ACD05}. We also thank an anonymous referee for many useful suggestions and corrections. We thank the project ``\'Emergence: Combinatoire \`a Paris'' funded by the city of Paris for supporting the visit of the first author during fall 2015. The first author is supported by the ERC-Advance grant 291092, ``Exploring the Quantum Universe'' (EQU), while the second author is supported by ANR-14-CE25-0014 (ANR GRAAL) and ANR-15-CE40-0013 (ANR Liouville).

\clearpage
\tableofcontents

\begin{center} \hrulefill  \quad \textit{From now on we fix once and for all the admissible} \hrulefill\\
	\hrulefill \quad\textit{critical non-generic weight sequence $\qseq$ as in \eqref{eq:asymptoticqk}. } \hrulefill  \end{center}

\section{Boltzmann planar maps and the lazy peeling process}
In this section we recall the edge-peeling process (also called the lazy peeling process) of \cite{Bud15}. We decided to rather mimic the presentation of  \cite{BCKscalingpeeling} in order for the reader to easily compare the differences between the present ``edge-peeling'' process and the ``face-peeling'' process used in \cite{BCKscalingpeeling,CLGpeeling}. We then study in more details two particular peeling algorithms that are designed to explore respectively the dual graph distance and the Eden distance on $B_{\infty}$.

\subsection{Enumeration}

\label{sec:enume}

If $ \map$ is a (rooted bipartite) planar map we denote by $\rootface \in \mathsf{Faces}( \map)$ the face adjacent on the right to the root edge. This face is called the root face of the map and its degree, denoted by $ \mathrm{deg}( \rootface)$, is called the perimeter of $\map$ (by parity constraint the perimeter of a bipartite map must be even). We write $| \map|$ for the number of vertices of $ \map$. For $\ell \geq 0$ and $n \geq 0$ we denote by $ \mathcal{M}^{(\ell)}_{n}$ the set of all (rooted bipartite) planar maps  of perimeter $2 \ell$ and with $n$ vertices, with the convention that $ \mathcal{M}^{(0)}_{1}$ comprises a single degenerate ``vertex planar map'' with no edges and a unique vertex. We put $ \mathcal{M}^{(\ell)}= \cup_{n \geq 1} \mathcal{M}^{(\ell)}_{n}$. Any planar map with at least one edge can be seen as a planar map with root face of degree $2$ by simply doubling the root edge and creating a root face of degree $2$. We shall implicitly use this identification many times in this paper.  We set 
 \begin{eqnarray} W^{(\ell)}_{n} = \sum_{ \map \in \mathcal{M}^{(\ell)}_{n}} \prod_{f \in \mathsf{Faces}( \map) \backslash \{ \rootface\}} q_{ \mathrm{deg}(f)/2} \quad \mbox{ and } \quad W^{(\ell)} = \sum_{n \geq 1} W^{(\ell)}_{n},   \label{eq:defWl}\end{eqnarray} where the dependence in $ \mathbf{q}$ is implicit as always in this paper. By convention $W_1^{(0)} = 1$ and $W_n^{(0)}=0$ for $n\geq 2$. The number $W^{(\ell)}$ can be understood as the partition function arising in the following probability measure: a $ \mathbf{q}$-Boltzmann planar map with perimeter $2\ell$ is a random planar map sampled according to the measure $ w( \cdot \mid \cdot \in \mathcal{M}^{(\ell)})$.
We now recall a few important enumeration results, see \cite[Eq. 3.15, Eq. 3.16]{BBG11}, \cite[Section 2]{LGM09} and \cite{Bud15}. Assuming that the weight sequence $q_{k} \sim c \ \kappa^{k-1} \ k^{-a}$ for $a \in (3/2;5/2)$ is fine-tuned (see  \cite[Section 2.2]{LGM09}) such that it is critical and admissible and satisfies the equation  \[\sum_{k=1}^\infty \binom{2k-1}{k-1} (4\kappa)^{1-k} q_{k} = 1-4\kappa,\]
then we have
 \begin{eqnarray} \label{eq:asymptowl} W^{(\ell)} \sim  \frac{c}{ 2\cos ( a \,\pi)} \kappa^{-\ell-1}\ell^{-a} \quad \mbox{ as }\ell \to \infty. \end{eqnarray}
Furthermore, from \cite[Corollary 2]{Bud15} we deduce that: 
\[ \frac{\kappa^\ell W_{n}^{(\ell)}}{\kappa W_{n}^{(1)}} \xrightarrow[n\to\infty]{}  2 \ell\  2^{-2\ell}\binom{2\ell}{\ell}.\]
The function $h^{\uparrow}(\ell) := 2 \ell\  2^{-2\ell}\binom{2\ell}{\ell}$, which does not depend on the weight sequence $ \mathbf{q}$, will play an important role in what follows in relation with a random walk whose step distribution we define now. Let $\nu$ be the probability measure on $ \mathbb{Z}$ defined by 
 \begin{eqnarray} \label{eq:defnu} \nu(k) = \left\{ \begin{array}{ll} q_{k+1} \kappa^{-k} & \mbox{for } k \geq 0 \\
2W^{(-1-k)} \kappa^{-k} & \mbox{for } k \leq -1 \end{array} \right..  \end{eqnarray}
Under our assumptions $\nu$ is indeed a probability distribution which has power-law tails. The function $ h^{\uparrow}$ is (up to a multiplicative constant) the only non-zero harmonic function on $\{1,2,3,...\}$ for the random walk with independent increments distributed according to $\nu$ (we say that $h^\uparrow$ is $\nu$-harmonic at these points) and that vanishes on $\{...,-2,-1,0\}$. This fact has been used in \cite{Bud15} to give an alternative definition of critical weight sequences:
\begin{propositiona}[\cite{Bud15}]\label{prop:qnubijection}
A weight sequence $\qseq$ is admissible and critical iff there exists a law $\nu$ on $ \mathbb{Z}$ such that $q_k = (\nu(-1)/2)^{k-1} \nu(k-1)$ and $h^\uparrow$ is $\nu$-harmonic on $\Z_{>0}$. In particular the random walk  with increments distributed according to $\nu$ oscillates (its $\limsup$ and $\liminf$ respectively are $+\infty$ and $-\infty$).
\end{propositiona}

\subsection{Edge-peeling process}
\subsubsection{Submaps in the primal and dual lattices}
Let $\map$ be a (rooted bipartite) planar map and denote by $\map^\dagger$ its dual map whose vertices are the faces of $\map$ and whose edges are dual to those of $\map$. The origin of $\map^\dagger$ is the root face $\rootface$ of $\map$. Let $ \mathfrak{e}^\circ$ be a finite connected subset of edges of $\map^\dagger$ such that the origin of $\map^\dagger$ is in $  \mathfrak{e}^\circ$, or more precisely incident to $ \mathfrak{e}^\circ$ (the letter ``e'' stands for explored). We associate to $ \mathfrak{e}^\circ$ a planar map $ \mathfrak{e}$ obtained roughly speaking by gluing the faces of $\map$ corresponding to the vertices in $ \mathfrak{e}^\circ$ along the (dual) edges of $  \mathfrak{e}^\circ$, see Fig.\,\ref{fig:surgery}. The resulting map, rooted at the root edge of $\map$, is a finite (rooted bipartite) planar map given with  several distinguished faces $ h_{1}, \ldots ,h_{k}  \in  \mathsf{Faces}(  \mathfrak{e})$ called the holes of $  \mathfrak{e}$ and corresponding to the connected components of $ \map^\dagger \backslash  \mathfrak{e}^\circ$. These faces are moreover simple meaning that there is no pinch point on their boundaries and that these boundaries do not share common vertices. We call such an object a planar map with holes. We say that $  \mathfrak{e}$ is a submap of $\map$ and write
\[ \mathfrak{e} \subset \map,\]
because $ \map$ can be obtained from $ \mathfrak{e}$ by gluing inside each hole $ h_{i}$ of $  \mathfrak{e}$ a bipartite planar map $ \mathfrak{u}_{i}$ of perimeter $ \mathrm{deg}( h_{i})$ ($ \mathfrak{u}$ stands for unexplored). To perform this operation we must assume that we have distinguished an oriented edge on the boundary of each hole $ h_{i}$ of $ \mathfrak{e}$ on which we glue the root edge of $ \mathfrak{u}_{i}$. We will not specify this further since these edges can be arbitrarily chosen using a deterministic procedure given $ \mathfrak{e}$. Notice that during this gluing operation it might be that several edges on the boundary  of a given hole of $ \mathfrak{e}$ get identified because the boundary of $ \mathfrak{u}_{i}$ may not be simple, see Fig.\,\ref{fig:surgery} below. However, this operation is rigid (see \cite[Definition 4.7]{AS03}) in the sense that given $ \mathfrak{e} \subset \map$ the maps $ (\mathfrak{u}_{i})_{1 \leq i \leq k}$ are uniquely defined. This definition even makes sense when $ \mathfrak{e}$ is a finite map and $ \map$ is an infinite map. Reciprocally, if $ \mathfrak{e} \subset \map$ is given, one can recover $ \mathfrak{e}^\circ$ the connected subset of edges of $ \map^\dagger$ as the set of dual edges between faces of $ \mathfrak{e}$ which are not holes.\medskip

The above discussion shows that there are two equivalent points of view on submaps of $ \map$: either they can be seen as connected subsets $\mathfrak{e}^\circ$ of edges of $\map^\dagger$ containing the origin, or as planar maps $ \mathfrak{e} \subset \map$ with (possibly no) holes that, once filled-in by proper maps, give back $\map$. In this paper, we will mostly work with the second point of view.

\begin{figure}[!h]
 \begin{center}
 \includegraphics[width=12cm]{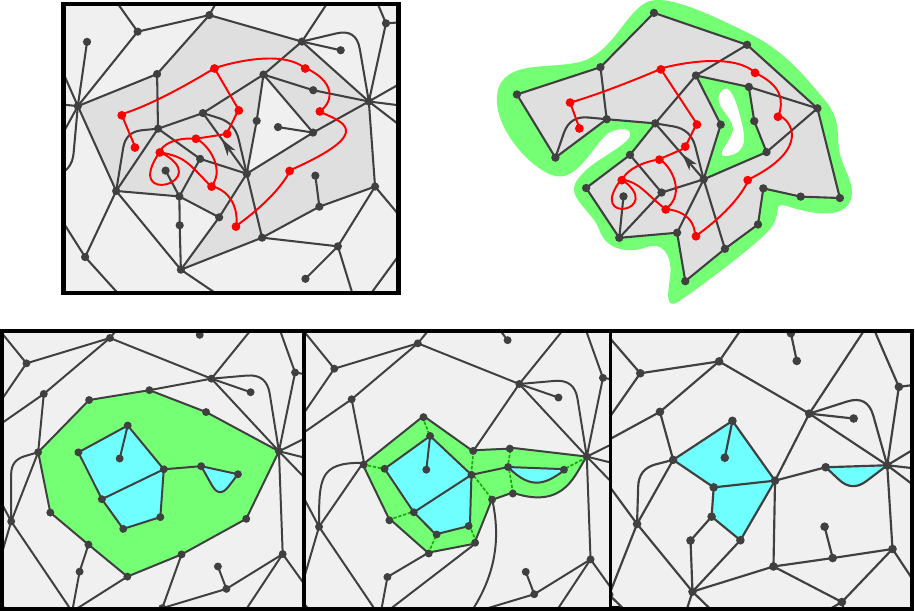}
 \caption{ \label{fig:surgery} Illustration of the duality between connected subsets of edges on the dual map and their associated submaps on the primal lattice. The gluing operation is illustrated below.}
 \end{center}
 \end{figure}

\subsubsection{Peeling exploration}
\label{sec:peelingexplo}
Suppose that $ \map$ is a (rooted bipartite) planar map. A branched edge-peeling exploration of $ \map$ is a sequence of increasing  submaps of $\map$ \[  \mathfrak{e}_{0} \subset \mathfrak{e}_{1} \subset \cdots \subset \mathfrak{e}_{n} \subset \cdots \subset \map,\] such that $ \mathfrak{e}_{i}$ is a planar map with holes whose number of inner edges, i.e. the ones not incident to a hole, is exactly $i \geq 0$, at least as long as the exploration has not stopped. The map $  \mathfrak{e}_{0}$ is made of a simple face of degree $ \mathrm{deg}(\rootface)$  corresponding to the root face $  \rootface$ of the map (recall that if necessary, one can always see a planar map as a map with root face degree $2$) and a unique hole of the same perimeter. Next, the exploration depends on an algorithm $ \mathcal{A}$ which associates to each map with holes $ \mathfrak{e}$ one edge $ \mathcal{A}( \mathfrak{e})$ on the boundary of one of its holes or the element $\dagger$ which we interpret as the will to stop the exploration.  This edge ``to peel'' $ \mathcal{A}( \mathfrak{e}_{i})$ tells us how to explore in $ \map$ in order to go from $ \mathfrak{e}_{i}$ to $ \mathfrak{e}_{i+1}$. More precisely, there are two cases:
\begin{itemize}
\item	 \textit{Case 1: }if the face on the other side of $\mathcal{A}( \mathfrak{e}_{i})$ corresponds to a new face in $ \map$ then $ \mathfrak{e}_{i+1}$ is obtained by adding to $ \mathfrak{e}_{i}$ the face adjacent to $  \mathcal{A}( \mathfrak{e}_{i})$ inside the corresponding hole of $ \mathfrak{e}_{i}$, see Fig.\,\ref{fig:peeling}.
\item  \textit{Case 2:} if the face on the other side of $ \mathcal{A}( \mathfrak{e}_{i})$ is already a face of $ \mathfrak{e}_{i}$ that means that $  \mathcal{A}( \mathfrak{e}_{i})$ is identified with another edge (necessarily adjacent to the same hole) in $ \map$.  Then $ \mathfrak{e}_{i+1}$ is obtained by performing this identification inside $ \mathfrak{e}_{i}$. This results  in splitting the corresponding hole in $ \mathfrak{e}_{i}$ yielding two holes in $ \mathfrak{e}_{i+1}$. The holes of perimeter $0$ are automatically filled-in with the vertex map, in particular the above process may close a hole which was made of two edges that have been identified in $ \map$, see Fig.\,\ref{fig:peeling}.
\end{itemize}
\begin{remark} At this point, the reader may compare the above presentation with that of \cite[Section 2.3]{BCKscalingpeeling} in order to understand the difference between the edge-peeling and the face-peeling processes. More precisely, when dealing with the face-peeling process the sequence $ \mathfrak{e}_{0} \subset \cdots \subset \mathfrak{e}_{i} \subset \cdots \subset \map$ of explored parts is again a sequence of maps with simple holes\footnote{with the slight difference that in \cite{BCKscalingpeeling} the holes can share vertices but not edges} but (unless the peeling has stopped), $ \mathfrak{e}_{i+1}$ is obtained from $ \mathfrak{e}_{i}$ by the addition of one face. Furthermore, in the case of the face-peeling process, $ \map$ is obtained from $ \mathfrak{e}_{i}$ by gluing maps with \emph{simple} boundary into the holes of $ \mathfrak{e}_{i}$.
\end{remark}
\begin{remark}
One can alternatively represent a peeling exploration $ \mathfrak{e}_{0} \subset \mathfrak{e}_{1} \subset \cdots \subset \map$ as the associate sequence of growing connected subset of edges $ (\mathfrak{e}_{i}^\circ)_{i \geq 0}$ on the dual map $\map^\dagger$ such that $ \mathfrak{e}_{i+1}^\circ$ is obtained from $ \mathfrak{e}_{i}^\circ$ by adding one edge of $\map^\dagger$ provided that connectedness is preserved (unless the exploration has stopped). We will however mostly use the first point of view.
\end{remark}

\begin{figure}[!h]
 \begin{center}
 \includegraphics[width=14cm]{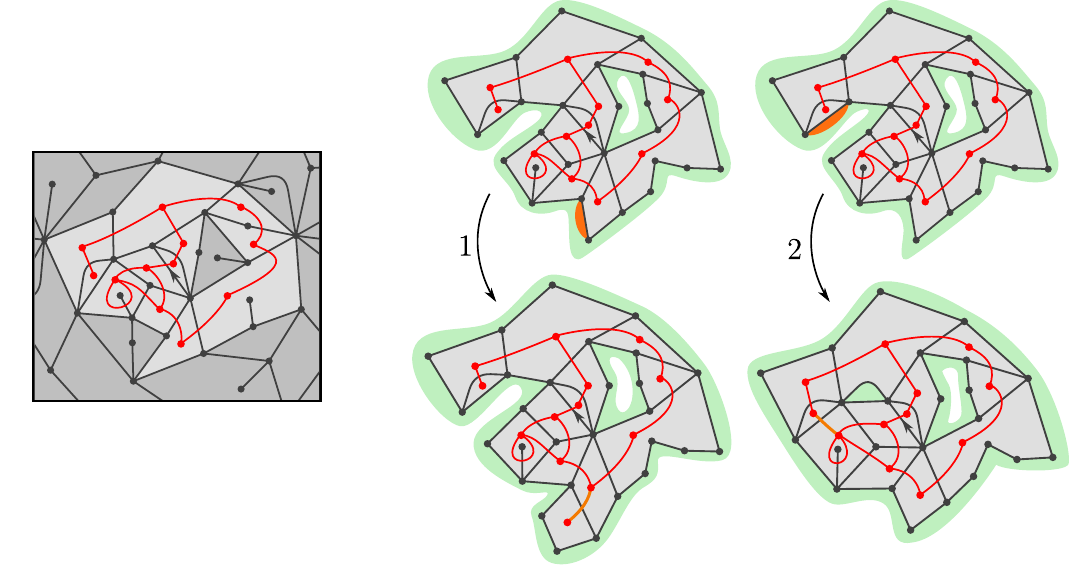}
 \caption{Illustration of the two cases which may happen when peeling an edge. In the first case, we add a face, in the second case we glue two edges of the boundary of a hole and create a new hole (possibly of perimeter $0$ when gluing two consecutive edges).}\label{fig:peeling}
 \end{center}
 \end{figure}

In the above branched edge-peeling exploration the evolving explored parts $( \mathfrak{e}_{i})_{i \geq 0}$ may have several holes. However in what follows we will restrict ourself to explorations of one-ended infinite maps $ \map$ (see \cite{BBCK16,BCKscalingpeeling} for the study of branched peeling explorations). In that case, we will fill-in all the holes of $ \mathfrak{e}_{i}$ whose associate component in $ \map$ is finite. That is, in \textit{case 2} above, when the hole of $ \mathfrak{e}_{i}$ is split into two new holes by the identification of two of its edges, we automatically fill-in the hole which is associated to a finite part in $ \map$. This gives rise to an exploration 
$ \mathfrak{e}_{0} \subset \cdots \subset \mathfrak{e}_{i} \subset \cdots \subset \map$ where $ \mathfrak{e}_{i+1}$ may have more than one inner edge on top of $ \mathfrak{e}_{i}$, but $ \mathfrak{e}_{i}$ always has a single hole on which we iteratively choose the edges to peel using the algorithm $ \mathcal{A}$. In the following, we will always consider such explorations and simply call them ``peeling explorations''. Let us now recall the results of \cite{Bud15}:

\begin{theorema}[\cite{Bud15}] \label{thm:budd15} Let $(P_{i},V_{i})_{i \geq 0}$ respectively be the half-perimeter of the unique hole and the number of inner vertices in a peeling exploration (with only one hole) of $ B_{\infty}$.  Then  $ (P_{i},V_{i})_{i \geq 0}$ is a Markov chain whose law  does not depend on the algorithm $ \mathcal{A}$ and is described as follows:
\begin{itemize}
\item the chain $(P_{i})_{i \geq 0}$ has the same law as $(W^\uparrow_{i})_{i\geq 0}$ the Doob $h^\uparrow$-transform of the random walk $(W)_{i \geq 0}$ started from $W_{0}=1$ and with i.i.d.\,independent increments of law $\nu$ given in \eqref{eq:defnu}. Equivalently, $(P_{i})_{i \geq 0}$ has the law of $(W_{i})_{i \geq 0}$ conditioned to never hit $ \mathbb{Z}_{\leq 0}$.
\item Conditionally on $(P_{i})_{i \geq 0}$ the variables $(V_{i+1}-V_{i})_{ i \geq 0}$ are independent and are distributed as the number of vertices in a $ \mathbf{q}$-Boltzmann planar map with perimeter $2(P_{i}-P_{i+1}-1)$ (where it is understood that this is $0$ when  $P_{i}-P_{i+1}-1 <0$).
\end{itemize}
\end{theorema}
In fact, the last theorem is still true if the peeling algorithm  $ \mathcal{A}$ is randomized as long as it does not use the information of the unexplored part at each peeling step. More precisely, conditionally on the current exploration $ \mathfrak{e}_{i}$, once we have selected an edge on the boundary of the hole of $ \mathfrak{e}_{i}$ independently of the remaining part of $ B_{\infty}$, assuming that the half-perimeter of this hole is $\ell \geq 1$, then the peeling of this edge leads to the discovery of a new face of degree $2k$ for $k\geq 1$ with probability 
 \begin{eqnarray} \label{eq:qp} p_{k}^{(\ell)}:=\nu(k-1) \frac{h^\uparrow(\ell+k-1)}{h^{\uparrow}(\ell)}. \end{eqnarray} 
 Otherwise this edge is identified with another edge of the boundary and the peeling swallows a bubble of length $2k$ for $0\leq k< \ell-1$ ($k=0$ corresponding to a bubble consisting of the single vertex-map) directly to the left of $\mathcal{A}(\mathfrak{e}_i)$ with probability
 \begin{eqnarray} \label{eq:qp2} p_{-k}^{(\ell)}:= \frac{1}{2}\nu(-k-1) \frac{h^\uparrow(\ell-k-1)}{h^{\uparrow}(\ell)}, \end{eqnarray}
 or to the right with the same probability. Notice that $\sum_{k=1}^\infty p_k^{(\ell)} + 2\sum_{k=0}^{\ell-2} p_{-k}^{(\ell)}=1$ is ensured precisely because $h^\uparrow$ is harmonic for the random walk $(W_i)_{i\geq 0}$. We will use many times below the fact that the probabilities of negative jumps for the process $(P)$ are uniformly dominated by those of $\nu$, more precisely since $h^\uparrow$ is non-increasing we have for $k \geq 1$ 
 \begin{eqnarray} \label{eq:bounddeltaperi} \mathbb{P}(\Delta P_{i} = -k \mid P_{i}= \ell ) = 2p_{-(k-1)}^{(\ell)} \leq \nu(-k) \leq C k^{-a},  \end{eqnarray} for some $C>0$ independent of $\ell \geq 1$ and $k \geq 1$.

 We now present two particular peeling algorithms that we will use in this work.

\subsection{Peeling by layers on the dual map $\map^\dagger$}
\label{sec:algodual}
It does not seem easy to use the edge-peeling process to systematically study the graph metric on $ B_{\infty}$ (this is because the degree of the faces are not bounded and so when discovering a new large face, one cannot \emph{a priori} know what is the distance to the root of all of its adjacent vertices). However, as in \cite{AB14,CLGpeeling} for the face-peeling process it is still possible to use the edge-peeling process in order to study the graph metric on the \emph{dual} of $ B_{\infty}$. Let us describe now the precise peeling process that we use for that.\medskip

Let $ \map$ be an infinite (rooted bipartite) one-ended planar map. We denote $ \map^\dagger$ the dual map of $\map$ and by $ \mathrm{d}_{ \mathrm{gr}}^\dagger$ the dual graph distance on $ \map^\dagger$. If $f \in \mathsf{Faces}(\map)$ the dual distance to the root face  $ \mathrm{d}_{ \mathrm{gr}} ^\dagger( f, \rootface)$ is called the \emph{height} of $f$ in $ \map$. The following peeling algorithm $ \mathcal{L}^\dagger$ is adapted to the dual graph distance (and fills the finite holes when created). Recall that the exploration starts with $ \mathfrak{e}_{0}$, the map made of a simple face of degree $\deg( \rootface)$ (and a unique hole of the same perimeter) which in the case of $B_{\infty}$ will be a $2$-gon after splitting the root edge as explained above. Inductively suppose that at step $i \geq 0$, the following hypothesis is satisfied:
\begin{center}
\begin{minipage}{14cm} $(H)$: There exists an integer $h \geq 0$ such that the explored map $ \mathfrak{e}_{i} \subset \map$ has a unique hole $ f^*$ such that all the faces adjacent to $f^*$ inside $ \mathfrak{e}_{i}$ are at height $h$ or $h+1$ in $\map$.  Suppose furthermore that the boundary edges of $f^*$ in $ \mathfrak{e}_{i}$ that are adjacent to faces at height $h$ form a connected part of the boundary of $ f^*$.
\end{minipage}
\end{center}
We will abuse notation and speak of the height of an edge of the boundary of the hole of $ \mathfrak{e}_{i}$ for the height of its incident face inside $ \mathfrak{e}_{i}$.  If $(H)$ is satisfied by $ \mathfrak{e}_{i}$ the next edge to peel $ \mathcal{L}^\dagger( \mathfrak{e}_{i} )$ is chosen as follows:
\begin{itemize}
\item If all edges incident to the hole $f^*$ of $ \mathfrak{e}_{i}$ are at height $h$ then $ \mathcal{L}^\dagger( \mathfrak{e}_{i})$ is any (deterministically chosen) edge on the boundary of $f^*$,
\item Otherwise $ \mathcal{L}^\dagger( \mathfrak{e}_{i})$ is the unique edge at height $h$ such that the edge immediately on its left is at height $h+1$.
\end{itemize}

\begin{figure}[!h]
 \begin{center}
 \includegraphics[width=10cm]{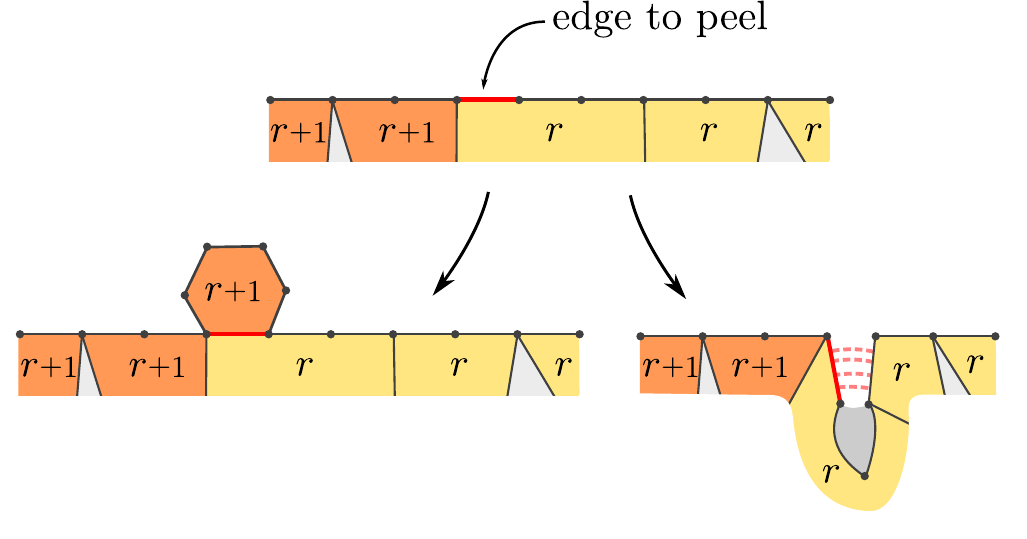}
 \caption{Illustration of the peeling using algorithm $ \mathcal{L}^\dagger$.}
 \end{center}
 \end{figure}
It is easy to check by induction that if one iteratively peels the edge determined by $\mathcal{L}^\dagger$ starting from $\mathfrak{e}_0$ then for every $i \geq 0$ the explored map $ \mathfrak{e}_{i}$ satisfies the hypothesis $(H)$ and therefore $\mathcal{L}^\dagger$ determines a well-defined peeling exploration of $\map$. Let us give a geometric interpretation of this peeling exploration. We denote by $ \mathsf{H}( \mathfrak{e}_{i})$ the minimal height in $\map$ of a face adjacent to the unique hole in $\mathfrak{e}_i$ and let $\theta_{r} = \inf\{ i \geq 0 : \mathsf{H}( \mathfrak{e}_{i}) = r\}$ for $r \geq 0$. On the other hand, for $r \geq 0$, we define by 
\[ \mathrm{Ball}_{r}^\dagger( \map),\]
the map made by keeping only the faces of $\map$ that are at height less than or equal to $r$ and cutting along all the edges which are adjacent on both sides to faces at height $r$ (see Fig.\,\ref{fig:dualball} for an example). Equivalently, the corresponding connected subset \[\left(\mathrm{Ball}_{r}^\dagger( \map)\right)^\circ\] of dual edges in $ \map^\dagger$ is given by those edges of $ \map^\dagger$ which contain at least one endpoint at height strictly less than $r$.
\begin{figure}[t]
	\begin{center}
		\includegraphics[width=12cm]{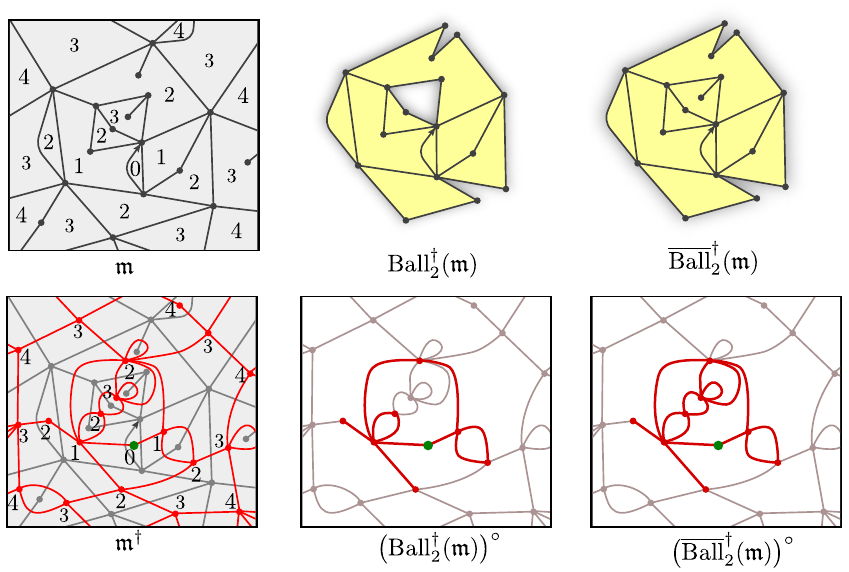}
		\caption{Example of a geodesic ball (and its hull) of radius $r=2$ with respect to the dual graph distance. \label{fig:dualball}}
	\end{center}
\end{figure}
 By convention we also put $ \mathrm{Ball}_{0}^\dagger(\map)$ to be the root face of $ \map$. Also, we write $ \overline{\mathrm{Ball}}_{r}^\dagger( \map)$ for the hull of these balls,  which are obtained by filling-in all the finite holes of $\mathrm{Ball}_{r}^\dagger( \map)$ inside $\map$ (recall that $\map$ is infinite and one-ended). After doing so, $ \overline{\mathrm{Ball}}_{r}^\dagger( \map)$ is a planar map with a single hole and we easily prove by induction on $r\geq 0$ that
 \begin{eqnarray} \label{eq:interpretation1} \mathfrak{e}_{\theta_{r}} = \overline{ \mathrm{Ball}}^\dagger_{r}( \map).  \end{eqnarray}
In the case when this edge-peeling exploration is performed on $ B_{\infty}$ we denote by $P_{i},V_{i}, H_{i}$ respectively the half-perimeter, the number of inner vertices and the minimal height of a face adjacent to the unique hole of $ \mathfrak{e}_{i}$ for $i \geq 0$.
\subsection{Eden model and Uniform peeling}
\label{sec:edenalgo}
We are using the same setup as in the previous section. Let $ \map$ be an infinite one-ended planar map. On the dual map $\map^\dagger$ of $\map$ we sample independent  weights $x_{e}$ for each edge $ e \in \mathsf{Edges}( \map^\dagger)$ distributed according to the exponential law $ \mathcal{E}(1)$ of mean 1, i.e. with density $e^{-x} \mathrm{d}x \mathbf{1}_{x >0}$. These weights can be used to modify the usual dual graph metric on $ \map^\dagger$ by considering the first-passage percolation distance: for $f_{1},f_{2} \in \mathsf{Faces}(\map)$
\[ \mathrm{d_{fpp}}( f_{1}, f_{2}) = \inf \sum_{ e \in \gamma} x_{e},\]
where the infimum is taken over all paths $\gamma : f_{1} \to f_{2}$ in the dual map $\map^\dagger$. This model (first-passage percolation with exponential edge weights on the dual graph) is often referred to as the Eden model on the primal map $\map$ \cite{AB14}. It is convenient in this section to view the edges of the map $ \map^\dagger$ as  real segments of length $x_{ e}$ for $e \in \mathsf{Edges}( \map^\dagger)$ glued together according to incidence relations of the map. This operation turns $\map^\dagger$ into a continuous length space (but we keep the same notation) and the distance $\mathrm{d_{fpp}}$ extends easily to all the points of this space. Now for $t >0$ we denote by \[\mathrm{Ball}^{ \mathrm{fpp}}_{t}( \map)\] the submap of $\map$ whose associated connected subset of dual edges $\left(\mathrm{Ball}^{ \mathrm{fpp}}_{t}( \map) \right)^\circ$ in $\map^\dagger$ is the set of all dual edges which have been fully-explored by time $t>0$, i.e.\,whose points (in the length space) are all at \textsc{fpp}-distance less than $t$ from the origin of $\map^\dagger$ (the root-face of $\map$). As usual, its hull $\overline{ \mathrm{Ball}}^{ \mathrm{fpp}}_{t}( \map)$ is obtained by filling-in the finite components  of its complement. It is easy to see that there are jump times $0=t_{0} < t_{1}<t_{2}< \cdots$ for this process and that almost surely (depending on the randomness of the $x_{e}$) the map $\overline{ \mathrm{Ball}}^{ \mathrm{fpp}}_{t_{i+1}}( \map)$ is obtained from $\overline{ \mathrm{Ball}}^{ \mathrm{fpp}}_{t_{i}}( \map)$ by the peeling of an appropriate edge (and by filling-in the finite component possibly created). The following proposition only relies on the randomness of the weights, the map $\map$ is fixed. 
\begin{proposition} If $\map$ is an infinite planar map with one end whose (dual) edges are endowed with i.i.d.\,exponential weights then we have:
\begin{itemize} \item the law of $( \overline{ \mathrm{Ball}}^{ \mathrm{fpp}}_{t_{i}}( \map))_{i \geq 0}$ is that of a uniform peeling on $\map$: conditionally on the past exploration, the next edge to peel is a uniform edge on the boundary of the explored part $\mathfrak{e}_{i}$;
\item conditionally on $( \overline{ \mathrm{Ball}}^{ \mathrm{fpp}}_{t_{i}}( \map))_{i \geq 0}$ the variables $t_{i+1}-t_{i}$ are independent and distributed as exponential variables of parameter given by the perimeter (that is twice the half-perimeter) of the explored part at time $i$.
\end{itemize}
\label{prop:peelingeden=uniform}
\end{proposition}
\proof Fix $\map$ and let us imagine the situation at time $t_{i}$ for $i \geq 0$. We condition on the sigma-field $ \mathcal{F}_{i}$ generated by all the exploration up to time $t_{i}$. Let us examine the edges in $\map^\dagger$ which are dual to the boundary of $  \mathfrak{e}_{i}=\overline{ \mathrm{Ball}}_{t_{i}}^{ \mathrm{fpp}}(\map)$. These come in two types: \emph{type-1} edges that are adjacent to a new face in the unexplored part (that is, if we peel one of those edges we are in case 1 of Section \ref{sec:peelingexplo}), and \emph{type-2} edges that link two faces adjacent to the boundary of the explored part (that is, if we peel one of these edges we are in case 2 of Section \ref{sec:peelingexplo}). See Fig.\,\ref{fig:eden-uniform}.

\begin{figure}[t]
 \begin{center}
 \includegraphics[width=12cm]{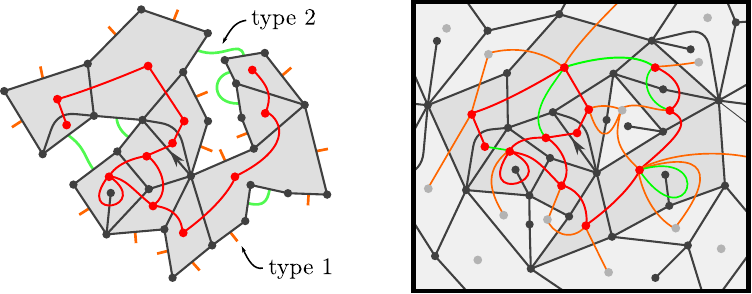}
 \caption{Illustration of the proof of Proposition \ref{prop:peelingeden=uniform}. The edges of the first type are in orange and those of the second type are in green. Regardless of their number and locations, the  next edge to peel can be taken unifomly on the boundary and the increase of time is given by an exponential variable of parameter given by the perimeter.}\label{fig:eden-uniform}
 \end{center}
 \end{figure}
 
Let us consider an edge $e^{(1)}$ of the first type and denote by $e^{(1)}_{-}$ its extremity in the explored region. Since this edge has not been fully explored at time $t_{i}$, it follows that its weight $x_{e^{(1)}}$ satisfies $x_{e^{(1)}} > t_{i}- \mathrm{d}_{ \mathrm{fpp}}(e^{(1)}_{-},  \rootface)$ and furthermore by properties of exponential variables conditionally on $ \mathcal{F}_{i}$
\[ y_{e^{(1)}}:= x_{e^{(1)}} -\big( t_{i}- \mathrm{d}_{ \mathrm{fpp}}(e^{(1)}_{-},  \rootface)\big)\] has the law $ \mathcal{E}(1)$ of an exponential variable of parameter $1$. Let us now examine the situation for an edge $e^{(2)}$ of the second type. We denote by $e^{(2)}_{-}$ and $e^{(2)}_{+}$ its endpoints. Since $e^{(2)}$ is being explored from both sides (in the length space representation) but has not been fully explored by time $t_{i}$, we have that $ x_{e^{(2)}} > \big( t_{i}- \mathrm{d_{fpp}}( e^{(2)}_{-}, \rootface) \big)+ \big( t_{i}- \mathrm{d_{fpp}}( e^{(2)}_{+}, \rootface) \big)$ and by the same argument as above conditionally on  $\mathcal{F}_{i}$
\[y_{ e^{(2)}}:=x_{e^{(2)}} -\big( t_{i}- \mathrm{d_{fpp}}( e^{(2)}_{-}, \rootface) \big)- \big( t_{i}- \mathrm{d_{fpp}}( e^{(2)}_{+}, \rootface) \big)\] is again exponentially distributed. Of course, an edge of the second type is dual to two edges of the boundary of $ \mathfrak{e}_{i}$. Apart from this trivial identification, the variables $y_{e}$ where $e$ runs over the edges dual to the boundary of $ \mathfrak{e}_{i}$ are, conditionally on $ \mathcal{F}_{i}$, independent of each other. Now, the time it takes until a new edge is fully explored is equal to 
\[t_{i+1}-t_{i} = \inf\{ y_{e} : e \mbox{ of the first type}\} \wedge  \frac{1}{2}\inf\{ y_{e} : e \mbox{ of the second type}\},\] where the factor $1/2$ again comes from the fact that edges of the second type are explored from both sides. By the above independence property, $t_{i+1}-t_{i}$ is thus distributed as an exponential variable of parameter
\[ t_{i+1}-t_{i} \overset{(d)}{=} \mathcal{E}( \#\{ \mathrm{edges \ of \ the \ first \ type}\} + 2 \# \{ \mathrm{edges \ of \ the \ second \ type}\}) = \mathcal{E}(2 \ell)\] where $2 \ell$ is the perimeter of the hole of $ \mathfrak{e}_{i}$. That proves the second part of the proposition. To see that conditionally on $ \mathcal{F}_{i}$ the next edge to peel  is uniform on the boundary, we may replace for each edge $e^{(2)}$   of the second type the variable $ \frac{1}{2}y_{e^{(2)}}$ of law $ \mathcal{E}(2)$ by the minimum of two independent exponential variables $\tilde{y}_{e_{1}^{(2)}}$ and $\tilde{y}_{e_{2}^{(2)}}$ of law $ \mathcal{E}(1)$ which we attach on the two edges dual to $ e^{(2)}$ on the boundary of $ \mathfrak{e}_{i}$. Finally, everything boils down to assigning to each edge of the boundary of the explored map an independent exponential variable of parameter $1$; the next edge to peel is the one carrying the minimal weight which is then uniform as desired. This completes the proof. \endproof

In the case when this edge-peeling exploration, also called the uniform peeling or Eden peeling, is performed on $ B_{\infty}$ we denote by $P_{i},V_{i},\tau_{i}$ for $i\geq0$ respectively the half-perimeter, the number of inner vertices and the jump times of the process $(\overline{ \mathrm{Ball}}_{t}^{ \mathrm{fpp}}( B_{\infty}))_{ t \geq 0}$. 

\section{Scaling limits for the perimeter and volume process}
\subsection{More on the perimeter process}
\label{sec:morehtransform}
Recall from Theorem \ref{thm:budd15} that the process of the half-perimeter $(P_{i})_{i\geq 0}$ of the only hole during an edge-peeling exploration of $B_{\infty}$ (which fills-in the finite holes) has the same law as $(W^\uparrow_{i})_{i \geq 0}$ the $h^\uparrow$-transform of the random walk $(W_{i})_{i \geq 0}$ started from $W_{0}=1$ whose critical step distribution $\nu$ is defined in \eqref{eq:defnu}.

 First, it is easy to see that the Markov chain $(P_{i})_{i \geq 0}$ or equivalently $(W^\uparrow_{i})_{i \geq 0}$ is transient. Indeed, if $T_{y}^\uparrow$ and $T_{y}$ denote the first hitting times of $y \in \mathbb{Z}_{>0}$ by respectively the chains $W^\uparrow$ and $W$, then we have 
 \begin{eqnarray}\label{eq:tooltransient} \mathbb{P}( T^\uparrow_{y}<\infty \mid W^\uparrow_{0} =p)=\frac{h^\uparrow(y)}{h^\uparrow(p)}\,  \mathbb{P}(W_k\geq 1, \;\forall k\leq T_y \mid W_{0} =p).  \end{eqnarray}
 Since $h^\uparrow$ is monotone (strictly) increasing on $\llbracket 1,\infty \llbracket$, the right-hand side is smaller than $1$ when $p>y$, hence $W^\uparrow$ is transient.

We now turn to estimating the expectation of $1/ W^\uparrow_{n}$. Those estimates will be crucial for the proofs of our main results. Recall that $h^\uparrow$ is $\nu$-harmonic on $\Z_{>0}$ and null on $ \mathbb{Z}_{ \leq 0}$. One can then consider the function $ h^\downarrow : \mathbb{Z} \to \mathbb{R}_{+}$ defined by 
 \begin{eqnarray} \label{eq:defhdown} h^\downarrow (\ell) = h^\uparrow(\ell+1) - h^{\uparrow}(\ell) = 2^{-2\ell}\binom{2\ell}{\ell}.  \end{eqnarray} Since $h^\uparrow$ is $\nu$-harmonic on $\{1,2,3,...\}$ it is not hard to see  that $h^\downarrow$ is $\nu$-harmonic on $\{1,2,3,...\}$ as well and satisfies furthermore $h^\downarrow(0) =1$. As for $h^\uparrow$, which gave us the conditioned walk $(W^\uparrow_{i})_{i \geq 0}$, one can consider the Markov process $(W^\downarrow_{i})_{i \geq 0}$ obtained as the Doob $h^\downarrow$-transform of the walk $(W_{i})_{i \geq 0}$ started from $W_{0}=p$. This process is easily seen (see \cite[Corollary 1]{Bud15}) to be the walk $W$ conditioned to hit $0$ before hitting $ \mathbb{Z}_{<0}$. For convenience we will set $W_i^\downarrow=0$ for all $i>j$ after its first hit of 0 at time $j$, which is almost surely finite due to the fact (Proposition A) that $W$ oscillates. We write $ \mathbb{P}_{p}$ and $ \mathbb{E}_{p}$ for the probability and expectation under which $W^\uparrow$ and $W^\downarrow$ are started from $p \geq 1$.

\begin{lemma} \label{lem:1/P} For any $p > 0$ and $n \geq 0$ we have
\begin{equation}\label{eq:Pupdowneq}
\mathbb{E}_p\left[ \frac{1}{W^\uparrow_n}\right] = \frac{\mathbb{P}_p(W^\downarrow_n > 0)}{p}.
\end{equation}
In particular, if $(P_{i})_{i \geq 0}$ is the half-perimeter process during an edge-peeling exploration of $B_{\infty}$ then 
\begin{equation}\label{eq:expinvP}
\mathbb{E}\left[ \frac{1}{P_n}\right] = 2\sum_{k=n+1}^\infty \frac{1}{k}\prob_1(W_k=0)  \quad \mbox{ and }\quad 
\sum_{n=0}^\infty \mathbb{E}\left[ \frac{1}{P_n}\right] = 2\sum_{k=1}^\infty \prob_1(W_k=0).
\end{equation}

\end{lemma}
\begin{proof}
The equality (\ref{eq:Pupdowneq}) follows directly from the definition of the $h^\uparrow$-transform and the exact forms of $h^\uparrow$ and $h^\downarrow$:
\begin{align}
\mathbb{E}_p\left[ \frac{1}{W^\uparrow_n}\right] &= \sum_{k=1}^\infty \frac{1}{k} \mathbb{P}_p(W^\uparrow_n = k) = \sum_{k=1}^\infty \frac{1}{k} \frac{h^\uparrow(k)}{h^\uparrow(p)}\mathbb{P}_p(W_i > 0 \text{ for }1\leq i<n,\, W_n = k)\nonumber \\
&= \frac{h^\downarrow(p)}{h^\uparrow(p)} \sum_{k=1}^\infty \frac{h^\uparrow(k)}{k h^\downarrow(k)} \mathbb{P}_p(W_n^\downarrow = k) = \frac{1}{p} \sum_{k=1}^\infty  \mathbb{P}_p(W_n^\downarrow = k)= \frac{1}{p}\mathbb{P}_p(W_n^\downarrow>0), \nonumber 
\end{align}
which gives the first claim. For the remaining statements it suffices to consider $p=1$.
Since $\inf\{ i : W_i^\downarrow = 0 \}$ is a.s. finite, we may identify
\begin{align}\label{eq:pdownid}
\mathbb{E}\left[ \frac{1}{P_n}\right]&=\mathbb{P}_1(W_n^\downarrow>0) = \sum_{j=n+1}^\infty \mathbb{P}_1(W^\downarrow_i > 0 \text{ for }1\leq i<j,\, W^\downarrow_j = 0)\nonumber\\
&= \frac{1}{h^\downarrow(1)}\sum_{j=n+1}^\infty \mathbb{P}_1(W_i > 0 \text{ for }1\leq i<j,\, W_j = 0).
\end{align}
We now use the cycle lemma to re-interpret the probabilities in the sum (see \cite[display before (1.7)]{ACD05}): For fixed $k>n\geq 0$ we can construct another sequence $(\tilde{W}_i)_{i\geq 0}$ by setting $\tilde{W}_i = 1+W_n-W_{n-i}$ for $i \leq n$, $\tilde{W}_{i} = W_n + W_k - W_{n+k-i}$ for $n<i<k$, and $\tilde{W}_i=W_i$ for $i \geq k$.
Then clearly $(\tilde{W}_i)_{i\geq 0}$ is equal in distribution to $(W_i)_{i\geq 0}$ while the event $W_i > 0$, $1\leq i<k$, $W_k = 0$, is equivalent to $\tilde{W}_k=0$ and the last maximum before time $k$ occurring at time $n$. Since the probability of the former event does not involve $n$ in its $W$-description, conditionally on $\tilde{W}_k=0$ the probability of the latter is equal for each $n\in\{0,1,\ldots,k-1\}$, and therefore \[\mathbb{P}_1(W_i > 0 \text{ for }1\leq i<k,\, W_k = 0) = \frac{1}{k} \prob_1(W_k=0).\]
Together with (\ref{eq:pdownid}) and $h^\downarrow(1)=1/2$ this implies the first equality in (\ref{eq:expinvP}), while the second one follows from interchanging the sums over $n$ and $k$.
\end{proof}

\subsection{Scaling limits for the perimeter}
\label{sec:perimeter}
We shall now study the scaling limit for the perimeter process. To avoid technical difficulties we exclude the case $a = 2$ which will be treated in a companion paper. Let $(S_{t})_{t \geq 0}$ be the $(a-1)$-stable Lévy process starting from $0$ with positivity parameter $ \rho = \mathbb{P}(S_{t} \geq 0)$ satisfying \[(a-1)(1-\rho)= \frac{1}{2}.\]
That is to say $(S_{t})_{t\geq 0}$ has no drift, no Brownian part and its Lévy measure has been normalized to 
 \begin{eqnarray} \label{eq:levykhintichine} \Pi( \mathrm{d}x) = \frac{ \mathrm{d}x}{x^{a}} \mathbf{1}_{x >0} + \frac{1}{\cos(\pi a)}\frac{ \mathrm{d}x}{|x|^{a}} \mathbf{1}_{x<0}.  \end{eqnarray}
It is then possible to define the process $(S^\uparrow_{t})_{t \geq 0}$ by conditioning $(S_{t})_{t \geq 0}$ to remain positive (see  \cite[Section 1.2]{CC08} for a rigorous definition). 

 \begin{proposition} \label{prop:scalingperimeter} If $a\in(3/2;2)\cup(2;5/2)$ we have the following convergence in distribution for the Skorokhod topology
\[ \left( \frac{W^\uparrow_{[nt]}}{n^{1/(a-1)}}\right)_{t \geq 0}  \xrightarrow[n\to\infty]{(d)}  \mathsf{p}_{ \mathbf{q}} \cdot  (S^\uparrow_{t})_{t \geq 0} \quad \mbox{ where} \quad \mathsf{p}_{ \mathbf{q}} = c^{1/(a-1)}.\]
 \end{proposition}
 \proof By the recent invariance principle \cite{CC08} it suffices to prove the convergence 
 in distribution 
\[ \frac{W_{n}}{n^{1/(a-1)}}  \xrightarrow[n\to\infty]{(d)}  \mathsf{p}_{ \mathbf{q}} \cdot  S_{1}.\]
First, it is easy to see from \eqref{eq:asymptoticqk} and \eqref{eq:asymptowl} that $\nu$ is a probability distribution in the domain of attraction of an $(a-1)$-stable law i.e.\,we have $a_{n}^{-1}W_{n}-b_{n}$ converges to an $(a-1)$-stable law for some scaling sequence $(a_{n})$ and centering sequence $(b_{n})$, see \cite[Theorem 8.3.1]{BGT89}. From the tail asymptotics of $\nu$ it follows that one can take $a_{n} = n^{1/(a-1)}$ and it remains to show that the centering sequences $(b_{n})$ can be set to $0$. This is always the case when $a\in (3/2,2)$ since no centering is needed; in the case when $a \in ( 2, 5/2)$ the fact that the random walk $(W_{i})_{i \geq 0}$ oscillates (Proposition \ref{prop:qnubijection}) implies that $\nu$ is centered and thus the centering sequence can be set to $0$ as well. In both cases, $a_{n}^{-1}W_{n}$ converges towards a strictly stable law whose limiting Lévy-Khintchine measure \eqref{eq:levykhintichine} is computed from the tails of $\nu$ by a straightforward calculation. 
\endproof 
\begin{remark} This scaling limit result should also hold true in the border case $a=2$ where the limit process is the symmetric Cauchy process without drift. Since we do not need this for our main results, which allude to either the dilute or the dense phase, we do not give the proof, which involves additional estimates to prove that the centering sequence can be set to $0$ in general. This can however be shown easily for the particular weight sequence $ q_{k} = 6^{1-k}/((2k-2)^2 -1)$ for $k > 1$ given in \cite[Eq. (80)]{Bud15} (see also Section \ref{sec:particular} below).
\end{remark}

Under the assumption of the above proposition, the local limit theorem \cite[Theorem 4.2.1]{IL71} implies that $\mathbb{P}_1(W_k=0) \sim C_0 k^{-1/(a-1)}$ as $k\to\infty$ for some $C_0>0$. Combining this with the first equation of \eqref{eq:expinvP} it follows that 
 \begin{eqnarray} \mathbb{E}\left[ \frac{1}{P_n}\right] \sim C n^{-1/(a-1)},  \label{eq:estim1/Pn}\end{eqnarray} for some constant $C>0$. See  \cite[Lemma 8]{CLGpeeling} for a similar estimate is the case of the face-peeling in random triangulations.

One can also deduce from the above proposition that any peeling exploration of $B_{\infty}$ will eventually discover the entire lattice (assuming further $a \ne 2$). The proof is mutatis mutandis the same as that of \cite[Corollary 7]{CLGpeeling} and reduces in the end to check that 
\[ \int_{1}^\infty \frac{ \mathrm{d}u}{(S_{u}^\uparrow)^{a-1}} = \infty \quad a.s.\]
 which can be proved using Jeulin's lemma.

\subsection{Scaling limits for the volume}
Our goal is now to study the scaling limit of the process $(V_{i})_{i \geq 0}$. We start with a result about the distribution of the size (number of vertices) of a $ \mathbf{q}$-Boltzmann planar map with a large perimeter, see  \cite[Proposition 6.4]{Ang03}, \cite[Proposition 9]{CLGpeeling} and \cite[Proposition 5]{Bud15} for similar statements in the case of more standard classes of planar maps. Recall that $a\in(3/2;5/2)$.\medskip

 Let $\xi_\bullet$ be a positive $1/(a-\frac12)$-stable random variable with Laplace transform
\begin{equation}
\expec [e^{-\lambda \xi_\bullet}] = \exp\left( - (\Gamma(a+1/2)\lambda)^{\frac{1}{a-1/2}}\right).
\end{equation}
Then $\expec [1/\xi_\bullet] = \int_{0}^\infty \mathrm{d}x \, \exp(-x^{1/(a-1/2)})/\Gamma(a+\frac12) = 1$ and we can define a random variable $\xi$ by biasing $\xi_{\bullet}$ by $x\to 1/x$, that is for any $f \geq 0$
\[ \mathbb{E}[f(\xi)] = \mathbb{E}\left[f( \xi_{\bullet}) \frac{1}{ \xi_{\bullet}}\right].\] 
Notice that $\xi$ has mean $\expec [\xi] = 1$. Recall that $| \map|$ denotes the number of vertices of a map $\map$.

\begin{proposition} \label{prop:volume} Suppose that $ \mathbf{q}$ is an admissible and critical weight sequence satisfying \eqref{eq:asymptoticqk}. Let $B^{(\ell)}$ be a $ \mathbf{q}$-Boltzmann planar map with root face degree $2\ell$ for $\ell \geq 1$.  Then we have 
 \begin{eqnarray} \mathbb{E} [| B^{(\ell)}|] \sim   \mathsf{b}_{ \mathbf{q}} \cdot {\ell^{a-1/2}} \quad \mbox{ as } \ell \to \infty \qquad \mbox{ where } \quad  \mathsf{b}_{ \mathbf{q}}=\frac{2\kappa\cos(\pi a)}{c\sqrt{\pi}}  \end{eqnarray}
and we have the convergence in distribution
\begin{equation}\label{eq:volumeconv}
\ell^{-a+\half} | B^{(\ell)}| \xrightarrow[\ell\to\infty]{\mathrm{(d)}} \mathsf{b}_{ \mathbf{q}}\cdot \xi.
\end{equation}

\end{proposition}
\begin{proof} Before entering the proof, let us introduce some convenient notation. A \emph{pointed} map $ \map_{\bullet}$ is a planar (rooted bipartite) map given with a distinguished vertex. We denote by $ \mathcal{M}^{(\ell)}_{\bullet}$ the set of all pointed finite planar maps of perimeter $2 \ell$ and define accordingly $W^{(\ell)}_{\bullet}$ as in \eqref{eq:defWl} after replacing $ \mathcal{M}^{(\ell)}$ by $ \mathcal{M}^{(\ell)}_{\bullet}$. With this notation in hand, it should be clear that 

\[\expec [| B^{(\ell)}|] =  \frac{W_\bullet^{(\ell)}}{W^{(\ell)}}.\]
It follows from \cite[Eq. (24)]{Bud15} that we have the exact expression $W_\bullet^{(\ell)}= \kappa^{-\ell} 2^{-2\ell}\binom{2\ell}{\ell}$. Combining this with \eqref{eq:asymptowl} we easily get the first statement of the proposition. To prove the second statement of the proposition one introduces $B^{(\ell)}_{\bullet}$, the pointed version of $ B^{(\ell)}$ whose law is given by $ w( \cdot \mid \cdot \in \mathcal{M}^{(\ell)}_{\bullet})$ and will first show that

\begin{equation}\label{eq:sizebiasedconv}
\ell^{-a+1/2}| B_\bullet^{(\ell)}| \xrightarrow[\ell\to\infty]{\mathrm{(d)}} \mathsf{b}_\qseq\, \xi_\bullet.
\end{equation}
This is sufficient to imply our claim, indeed if $ \phi : \mathbb{R}_{+}^*  \to \mathbb{R}_{+}$ is a bounded continuous function with compact support in $ \mathbb{R}_{+}^*$ we have 
 \begin{eqnarray*}
  \mathbb{E}\left[  \phi \left( \ell^{-a+1/2} |B^{(\ell)}| \right) \right]\!\!\! &=&   \mathbb{E}\left[  \phi \left( \ell^{-a+1/2} |B^{(\ell)}_{\bullet}| \right) / |B^{(\ell)}_{\bullet}| \right] / \mathbb{E}\left[ 1/|B^{(\ell)}_{\bullet}| \right] \\ 
  &=&  \mathbb{E}\left[  \phi \left( \ell^{-a+1/2} |B^{(\ell)}_{\bullet}| \right) / (\ell^{-a+1/2}|B^{(\ell)}_{\bullet}|) \right] \cdot  \mathbb{E}\left[ \ell^{-a+1/2}|B^{(\ell)}| \right]\\
  & \xrightarrow[\ell\to\infty]{}& \mathbb{E}[ \phi(\mathsf{b}_\qseq\, \xi_\bullet) / (\mathsf{b}_\qseq\, \xi_\bullet) ] \cdot \mathsf{b}_{ \mathbf{q}} = \mathbb{E}[ \phi(\mathsf{b}_\qseq\, \xi)]   \end{eqnarray*} where in the last line the convergence is obtained after remarking that $\phi(x)/x$ is bounded and continuous because $\phi$ has compact support in $ \mathbb{R}_{+}^*$. This indeed proves the desired convergence in distribution. 
  
  We now turn to proving \eqref{eq:sizebiasedconv} using Laplace transforms. In this part we highlight the dependence in $ \mathbf{q}$ since it is crucial in the calculation and write $W^{(\ell)}( \mathbf{q})$ for $W^{(\ell)}$, $w_{ \mathbf{q}}$ for $w$, etc. Recall that $|\map|$ denotes the number of vertices of a map $\map$ and let us introduce for $g\in[0,1]$ the generating function
\[
W^{(\ell)}_{\bullet}(g;\qseq) := \sum_{\map_\bullet\in\maps_\bullet^{(\ell)}} w_\qseq(\map_\bullet) \, g^{|\map_\bullet|},
\]
such that $W^{(\ell)}_{\bullet}(g;\qseq)$ is strictly increasing on $g\in[0,1]$ and $W^{(\ell)}_{\bullet}(1;\qseq) = W^{(\ell)}_\bullet(\qseq) <\infty$. With this notation we have for all $\lambda >0$
 \begin{eqnarray} \label{eq:laplacetransform}\mathbb{E}[\exp(-\lambda |B^{(\ell)}_{\bullet}|)] = \frac{W_{\bullet}^{(\ell)}(e^{-\lambda}; \mathbf{q})}{W_{\bullet}^{(\ell)}(\mathbf{q})}.  \end{eqnarray}
Using Euler's formula we can rewrite this as $W^{(\ell)}_{\bullet}(g;\qseq) = g^{1+\ell} W_\bullet^{(\ell)}(\qseq_g)$ where $\qseq_g$ is the weight sequence determined by $(q_g)_k := g^{k-1} q_k$ for $k\geq 1$.
Since $\qseq_g$ is necessarily an admissible weight sequence we know that $W_\bullet^{(\ell)}(\qseq_g) = \kappa_g^{-\ell} h^\downarrow(\ell)$ for some $\kappa_g>0$, where $h^\downarrow$ is defined in \eqref{eq:defhdown}. According to \cite{MM07} we have $\kappa_g = 1/(4\bar{x})$ where $\bar{x}$ is the unique positive solution to $f_{\qseq_g}(\bar{x}) = 1 - \frac{1}{\bar{x}}$ with 
\begin{equation*}
f_{\qseq}(x) := \sum_{k=1}^\infty x^{k-1} \binom{2k-1}{k} q_k. 
\end{equation*} 
Since $f_{\qseq_g}(\bar{x}) = f_\qseq(g\bar{x})$, this is equivalent to $\kappa_g = g/(4x)$ with $x\in (0, 1/(4\kappa) )$ the unique positive solution to $f_\qseq(x)=1-\frac{g}{x}$, or better $\bar{f}_\qseq(x) = g$ with $\bar{f}_\qseq(x) := x(1-f_\qseq(x))$.

Our weight sequence $\qseq$ is chosen exactly such that $\bar{f}_\qseq( 1/(4\kappa)) = 1$ and $\bar{f}_\qseq '(1/(4\kappa)) = 0$.
Since $q_k \sim c \kappa^{k-1} k^{-a}$ as $k\to\infty$ we find that
\begin{equation*}
\bar{f}_\qseq(x) \sim 1 - \frac{c\,\Gamma(\half-a)}{2\kappa\sqrt{\pi}}(1-4\kappa x)^{a-\half}=1 - \frac{1}{\Gamma(a+\half)\mathsf{b}_\qseq}(1-4\kappa x)^{a-\half}\quad\quad \text{as }x\nearrow \frac{1}{4\kappa}.
\end{equation*}
It follows that 
\begin{equation} \label{eq:estimkappag}
\frac{g\kappa}{\kappa_g} = 4\kappa x \sim 1-\left(\Gamma(a+\half)\mathsf{b}_\qseq(1-g)\right)^{1/(a-\half)}\quad\quad\text{ as }g\nearrow 1.
\end{equation}
Using that $W_\bullet^{(\ell)}(g;\qseq)/W_\bullet^{(\ell)}(\qseq) = g \left(g\kappa/\kappa_g\right)^{\ell}$ and setting $g = \exp(-\lambda \ell^{\half-a})$ with $\lambda>0$ we find 
\begin{eqnarray}
 \lim_{\ell \to \infty}\mathbb{E}\left[\exp\left(-\lambda \ell^{ \half-a} |B^{(\ell)}_{\bullet}|\right)\right] &\underset{ \eqref{eq:laplacetransform}}{=}& \lim_{\ell \to \infty} \frac{ W_\bullet^{(\ell)}(\exp(-\lambda \ell^{\half-a});\qseq)}{W_\bullet^{(\ell)}(\qseq)}\nonumber\\ & \underset{\eqref{eq:estimkappag}}{=}& \lim_{\ell\to\infty} \left( 1-\frac{1}{\ell}\left(\Gamma(a+\half)\mathsf{b}_\qseq\lambda\right)^{1/\left(a-\half\right)}\right)^{\ell}\nonumber\\  &=& \exp\left( - \left(\Gamma(a+\half)\mathsf{b}_\qseq\lambda\right)^{1/\left(a-\half\right)}\right)\nonumber\\
 & =& \expec [\exp(- \lambda\mathsf{b}_\qseq \xi_\bullet)]\nonumber
\end{eqnarray}
thereby proving the convergence (\ref{eq:sizebiasedconv}).
\end{proof}

\begin{remark} In this work, the number of vertices of the primal map $B_{\infty}$ (or, equivalently, the number of faces of $B_{\infty}^\dagger$) has been taken as the notion of volume. Actually, all the results on the volume could be translated in terms of number of faces of $B_{\infty}$ (or vertices of $B_{\infty}^\dagger$) up to changing the constant $ \mathsf{b}_{ \mathbf{q}}$. More precisely, the proposition above and its consequences in the paper hold true if one uses $\|\map\|$, the number of faces of the map $\map$, instead of $|\map|$ and a new constant 
\[ \mathsf{b}_{ \mathbf{q}}^F = \left(\frac{1}{4\kappa}-1\right)\mathsf{b}_{\mathbf{q}} = (1-4\kappa)\frac{\cos(\pi a)}{2c\sqrt{\pi}}\]
instead of $\mathsf{b}_{\mathbf{q}}$.
This can be proved either by generating function techniques as above (see \cite{BudOn}) or by probabilistic representation of the volume using the Bouttier-Di Francesco-Guitter encoding (see \cite{CCMOn}).
\end{remark}
We are now able to introduce the scaling limit for the perimeter and volume process during a peeling exploration of $ B_{\infty}$. Recall from Section \ref{sec:perimeter} the definition of $ (S^\uparrow_{t})_{t \geq 0}$ as the $(a-1)$-stable Lévy process conditioned to survive. We let $\xi_{1}, \xi_{2}, \ldots $ be a sequence of  independent real random variables distributed as the variable $\xi$ of Proposition \ref{prop:volume}. We assume that this sequence is independent of the process $(S_{t}^\uparrow)_{t \geq 0}$ and for every $t \geq 0$ we set   \begin{eqnarray} \label{eq:defzt} Z_{t} = \sum_{t_{i} \leq t} \xi_{i} \cdot |\Delta S_{t_{i}}^\uparrow|^{a- \frac{1}{2}} \mathbf{1}_{ \Delta S^\uparrow_{t_{i}} < 0},  \end{eqnarray} where $t_{1}, t_{2}, \ldots$ is a measurable enumeration of the jump times of $S^\uparrow$. Since $ x \mapsto x^{a- \frac{1}{2}} \mathbf{1}_{ x <0}$ integrates the Lévy measure of $(S_{t})_{t \geq 0}$ in the neighborhood of $0$ it is easy to check  that $(Z_{t})_{t \geq 0}$ is a.s.~finite for all $t\geq 0$. The analog of \cite[Theorem 1]{CLGpeeling} and \cite[Theorem 3]{Bud15} is

\begin{theorem} \label{thm:scalingperimeter+volume}Let $(P_{i},V_{i})_{i \geq 0}$ respectively be the half-perimeter and the number of inner vertices in a peeling exploration of $ B_{\infty}$. For $a \ne 2$ we have the following convergence in distribution in the sense of Skorokhod
\[ \left( \frac{P_{[nt]}}{n^{\frac{1}{a-1}}}, \frac{V_{[nt]}}{n^{ \frac{a-1/2}{a-1}}} \right)_{t \geq 0}  \xrightarrow[n\to\infty]{(d)}  \left( \mathsf{p}_{ \mathbf{q}} \cdot S^\uparrow_{t},  \mathsf{v}_{ \mathbf{q}} \cdot  Z_{t}\right)_{t \geq 0},\]
where  $\mathsf{v}_{ \mathbf{q}} = \mathsf{b_\qseq} (\mathsf{p}_{\mathbf{q}})^{a-1/2}$ and $ \mathsf{p}_{ \mathbf{q}}$ and $ \mathsf{b}_{ \mathbf{q}}$ are as in Propositions \ref{prop:scalingperimeter} and \ref{prop:volume}.
\end{theorem}

\begin{figure}[t]
 \centering
 \includegraphics[width=\linewidth]{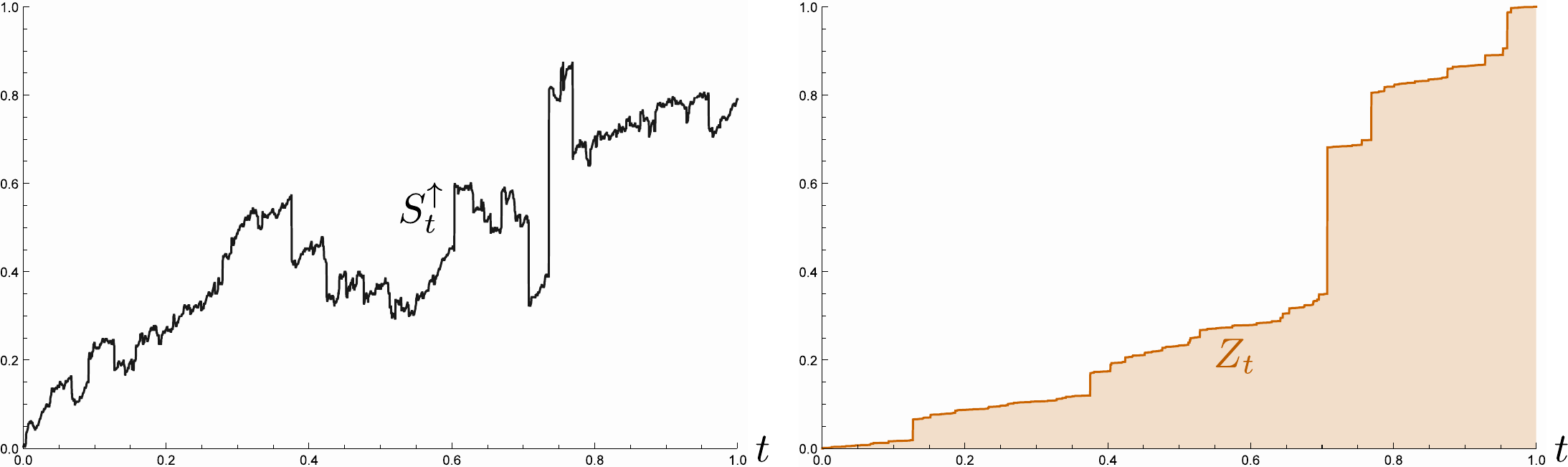}
 \caption{Simulation of the processes $ S^\uparrow$ and $Z$ when $a=2.3$.}
 \end{figure}
\proof The proof is the same as that of Theorem 1 of \cite{CLGpeeling} with the appropriate updates. The convergence of the first component is given
by  Proposition \ref{prop:scalingperimeter}, it remains to
study the   conditional distribution of the second component given 
the first one. Recall that the number of inner vertices in $ \mathfrak{e}_{n}$ can be written as 
\[V_n= \sum_{ i = 0}^{n-1}\mathcal{X}^{(i)}_{P_{i}-P_{i+1} -1},\] where $ \mathcal{X}^{(i)}_{j}$ for $i \geq 0$ and $j \in \mathbb{Z}$ are independent random variables such that $ \mathcal{X}^{(i)}_{j}$ has the same distribution as the number of vertices inside a $ \mathbf{q}$-Boltzmann random map with perimeter $2j$ if $j \geq 0$ and is $0$ otherwise. To simplify notation we use the notation $\wt \Delta P_{i} = P_{i}- P_{i+1}-1$ below. Fix $\ve >0$ and set for $k \in \{1,2,... , n\}$\begin{equation}
\label{holes-filled2}
 V_k^{> \ve}=\sum_{i=0}^{k-1}  \mathcal{X}_{\wt \Delta P_{i}}^{(i)}\mathbf{1}_{ \wt \Delta P_{i} > \ve n^{1/(a-1)}}
\;,\quad  V_k^{\leq \ve}=\sum_{i=0}^{k-1} \mathcal{X}_{\wt \Delta P_{i}}^{(i)}\mathbf{1}_{ 0 \leq \wt \Delta P_{i} \leq \ve n^{1/(a-1)}}.
\end{equation}
It is then easy to combine Proposition \ref{prop:volume} and \eqref{eq:bounddeltaperi} in order to deduce (see the proof of  \cite[Theorem 1]{CLGpeeling} for the detailed calculation) that 
\begin{equation}
\label{petit-saut}
n^{-(a-1/2)/(a-1)}{\E}[V_{n}^{ \leq \varepsilon} ] \leq C\,\sqrt{\ve},
\end{equation}
for some $C>0$ independent of $n$ and $ \varepsilon$.

On the other hand, by Proposition \ref{prop:scalingperimeter} and the fact that $(S^\uparrow_{t})_{t \geq 0}$ does not have jumps of size exactly $ - \varepsilon$ almost surely, it follows that jointly with the convergence of the first component in the theorem we have the following convergence in distribution for the Skorokhod topology (see \cite[Proof of Theorem 6]{CLGpeeling} for details) 
  \begin{eqnarray} \label{eq:convverszeps} \left(n^{ -\frac{a-1/2}{a-1}} \cdot V^{> \varepsilon}_{[nt]} \right)_{t \geq 0}  \xrightarrow[n\to\infty]{(d)}  \left(  \mathsf{v}_{ \mathbf{q}} \cdot  Z^{ > \varepsilon}_{t}\right)_{t \geq 0},  \end{eqnarray} where the process $(Z^{ > \varepsilon}_{t})$ is defined as $(Z_{t})$ but only keeping the negative jumps of $ (S^{\uparrow}_{t})$ of absolute size larger than $   \varepsilon/ \mathsf{p}_{ \mathbf{q}}$. Then, it is easy to verify that, for every $\delta>0$ and any $t_{0}>0$ fixed we have 
\[ \mathbb{P}\Big(\sup_{0\leq t\leq t_{0}} |Z_t-Z^{>\ve}_t| >\delta\Big) 
\xrightarrow[\ve\to 0]{} 0.\] We can use the last display, together with \eqref{eq:convverszeps} and \eqref{petit-saut} to deduce the desired convergence in distribution. 
\endproof
\section{The dilute phase}
\begin{center} \hrulefill  \quad \textit{In this section we suppose that $a \in \left(2, \frac52\right)$ } \hrulefill  \end{center}
In this section, we study the geometry of $B_{\infty}^\dagger$ both for the dual graph distance $ \mathrm{d}_{ \mathrm{gr}}^\dagger$ and the first-passage percolation distance $ \mathrm{d_{fpp}}$ in the dilute phase $a \in (2, \frac52)$. Our main results are Theorem \ref{thm:scalinglayers} and Proposition \ref{prop:scalingeden}. The proofs in this section are similar to those of \cite{CLGpeeling} and only the main differences are highlighted. The key idea is to relate the growth of the distances along the peeling process to the perimeter process via a time change. We start with the Eden model which is much simpler.
\subsection{Eden model}
\begin{proposition}[Distances in the uniform peeling]  \label{prop:scalingeden} Let $(P_{i},V_{i}, \tau_{i})_{i \geq 0}$ respectively be the half-perimeter, the number of inner vertices and the times of jumps of the exploration process in the uniform peeling of $B_{\infty}$ as described in Section \ref{sec:edenalgo}. Then we have the following convergence in distribution for the Skorokhod topology 
 \begin{eqnarray*} \left( \frac{P_{[nt]}}{ n^{\frac{1}{a-1}}}, \frac{V_{[nt]}}{n^{ \frac{a- 1/2}{a-1}}}, \frac{ \tau_{[nt]}}{n^{\frac{a-2}{a-1}}}\right)_{t \geq 0}  &\xrightarrow[n\to\infty]{(d)}&  \left( \mathsf{p}_{  \mathbf{q}} \cdot S_{t}^\uparrow, \mathsf{v}_{ \mathbf{q}} \cdot Z_{t}, \frac{1}{ 2\mathsf{p}_{ \mathbf{q}}} \cdot \int_{0}^t \frac{ \mathrm{d}u}{S_{u}^\uparrow} \right)_{t \geq 0}.  \end{eqnarray*}
\end{proposition}
The above result can easily be translated in geometric terms. Recall the notation $  \mathrm{Ball}_{r}^{ \mathrm{fpp}}( B_{\infty})$ from Section	\ref{sec:edenalgo}. We denote by $ | \overline{\mathrm{Ball}}_{r}^{ \mathrm{fpp}}( B_{\infty})|$ and $|\partial \overline{\mathrm{Ball}}_{r}^{ \mathrm{fpp}}( B_{\infty})|$ respectively the size (number of inner vertices) and the half-perimeter of the unique hole of $ \overline{\mathrm{Ball}}_{r}^{ \mathrm{fpp}}( B_{\infty})$. Then from the geometric interpretation of Section \ref{sec:edenalgo} and the above result we have the following convergence in distribution in the sense of Skorokhod
 \begin{eqnarray} \label{eq:interprescalingeden}
 \left( \frac{ | \partial  \overline{\mathrm{Ball}}_{[tn]}^{ \mathrm{fpp}}(B_{\infty})|}{n^{ \frac{1}{a-2}}}, \frac{| \overline{\mathrm{Ball}}_{[tn]}^{ \mathrm{fpp}}(B_{\infty})|}{n^{ \frac{a-1/2}{a-2}}}\right)_{t \geq 0} &\xrightarrow[n\to\infty]{(d)}& \left( \mathsf{p}_{ \mathbf{q}} \cdot S^\uparrow_{\vartheta_{2 \mathsf{p}_{ \mathbf{q}}t}},  \mathsf{v}_{ \mathbf{q}} \cdot Z_{\vartheta_{2 \mathsf{p}_{ \mathbf{q}}t}}\right)_{t \geq 0},
 \end{eqnarray}
 where for $t\geq0$ we have put $\vartheta_{t} = \inf\{ s \geq 0 : \int_{0}^s \frac{ \mathrm{d}u}{ S^\uparrow_{u}} \geq t\}$. In the work \cite{CLGpeeling}, the process $S^\uparrow_{\vartheta_{t}}$ (called the first Lamperti transform of $S^\uparrow$) could be interpreted as a reverse branching process, but this is not the case anymore here since our Lévy processes now have positive and negative jumps (Lamperti representation theorem links branching processes to Lévy processes with only negative jumps). 
\proof[Proof of Proposition \ref{prop:scalingeden}] Here also, the proof is the same as that of Theorem 4 of \cite{CLGpeeling} with the appropriate updates. The joint convergence of the first two components is given by Theorem  \ref{thm:scalingperimeter+volume}. We now prove the convergence of the
third component jointly
with the first two. Recall from Proposition \ref{prop:peelingeden=uniform} that conditionally on $(P_{i},V_{i})_{i \geq 0}$ we have
\[ \tau_{n} = \sum_{i=0}^{n-1} \frac{ \mathbf{e}_{i}}{2P_{i}},\] where $ \mathbf{e}_{ i}$ are independent exponential variables of expectation $1$. Using Proposition \ref{prop:scalingperimeter} and an easy law of large number we deduce that for every $ \varepsilon>0$ we have the following convergence 
 \begin{eqnarray} \left(n^{- \frac{a-2}{a-1}} \big(\tau_{[nt]}-\tau_{[n \varepsilon]}\big)\right)_{ t \geq  \varepsilon} \xrightarrow[n\to\infty]{(d)} \left(\frac{1}{ 2\mathsf{p}_{ \mathbf{q}}} \int_{ \varepsilon}^t \frac{ \mathrm{d}u}{S^\uparrow_{u}}\right)_{ t \geq \varepsilon}, \label{eq:convinteps}  \end{eqnarray} and this convergence holds jointly with the first two components considered in the proposition (see \cite[Proof of Theorem 4]{CLGpeeling} for the details of the calculation needed). Hence, to finish the proof of the proposition, it suffices to see that for any $ \delta>0$ we have 
\[ \lim_{ \varepsilon \to 0} \sup_{ n \geq 1} \mathbb{P}\left( n^{ -\frac{a-2}{a-1}} \cdot \tau_{[n \varepsilon]} > \delta \right) =0 \quad \mbox{ and } \quad \lim_{ \varepsilon \to 0} \mathbb{P}\left(\int_{0}^ \varepsilon \frac{ \mathrm{d}u}{S^\uparrow_{u}} >  \delta \right)=0.\]
For the first limit, we use \eqref{eq:estim1/Pn} to get 
\[ \mathbb{E}[\tau_{[n \varepsilon]}] = \left.\mathbb{E}\left[ \mathbb{E}\left[ \sum_{i=0}^{[ n \varepsilon]} \frac{ \mathbf{e}_{i}}{2P_{i}}\  \right| (P_{i})_{i \geq 0}\right]\right] = \sum_{i=0}^{[ n \varepsilon]}  \mathbb{E}\left[\frac{1}{2P_{i}}\right] \underset{ \eqref{eq:estim1/Pn}}{\leq} C ( \varepsilon n)^{\frac{a-2}{a-1}},\] for some constant $C>0$. The desired result follows from an application of Markov's inequality. The second statement just follows from the fact that $(S^\uparrow_{t})^{-1}$ is almost surely integrable around $0^+$ since $a>2$. One cheap way to see this is to take expectations in \eqref{eq:convinteps} and using Fatou's lemma together with the last calculation to get 
\[ \frac{1}{2 \mathsf{p}_{ \mathbf{q}}} \mathbb{E}\left[ \int_{ \varepsilon}^1  \frac{ \mathrm{d}u}{S^\uparrow_{u}} \right] \leq C(1 + \varepsilon^{ \frac{a-2}{a-1}}).\]
Sending $ \varepsilon \to 0$ we deduce that indeed $(S^\uparrow_{t})^{-1}$ is almost surely integrable around $0$.
 \endproof 
\subsection{Dual graph distance}

\label{sec:dilutelayers}

\begin{theorem}[Distances in the peeling by layers]  \label{thm:scalinglayers} Let $(P_{i},V_{i}, H_{i})_{i \geq 0}$ respectively be the half-perimeter, the number of inner vertices and the minimal height of a face adjacent to the hole of $ \mathfrak{e}_{i}$ in the peeling of $B_{\infty}$ using algorithm $ \mathcal{L}^\dagger$. Then we have the following convergence in distribution for the Skorokhod topology 
 \begin{eqnarray*} \left( \frac{P_{[nt]}}{ n^{\frac{1}{a-1}}}, \frac{V_{[nt]}}{n^{ \frac{a- 1/2}{a-1}}}, \frac{ H_{[nt]}}{n^{\frac{a-2}{a-1}}}\right)_{t \geq 0}  &\xrightarrow[n\to\infty]{(d)}&  \left( \mathsf{p}_{  \mathbf{q}} \cdot S_{t}^\uparrow, \mathsf{v}_{ \mathbf{q}} \cdot Z_{t},  \mathsf{h}_{ \mathbf{q}} \cdot \int_{0}^t \frac{ \mathrm{d}u}{S_{u}^\uparrow} \right)_{t \geq 0},  \end{eqnarray*}
 where $\mathsf{h}_{ \mathbf{q}} = \mathsf{a}_{ \mathbf{q}}/(2\mathsf{p}_{ \mathbf{q}})$ and $ \mathsf{a}_{ \mathbf{q}}$ is defined below by \eqref{analoProp13}.
\end{theorem}
Let us again give a more geometric interpretation of the above result. Recall from \eqref{eq:interpretation1}
 that the peeling process using algorithm $ \mathcal{L}^\dagger$ discovers balls for the dual graph distance on $ B_{\infty}$ and we denote by $ | \overline{\mathrm{Ball}}_{r}^{\dagger}( B_{\infty})|$ and $|\partial \overline{\mathrm{Ball}}_{r}^{ \dagger}( B_{\infty})|$ respectively the size (number of inner vertices) and the half-perimeter of its unique hole of the hull of the ball of radius $r$ for the dual distance. Then with the same notation as in \eqref{eq:interprescalingeden} the above result implies the convergence in distribution in the sense of Skorokhod
 \begin{eqnarray}
 \left( \frac{ | \partial  \overline{\mathrm{Ball}}_{[tn]}^{ \dagger}(B_{\infty})|}{n^{ \frac{1}{a-2}}}, \frac{| \overline{\mathrm{Ball}}_{[tn]}^{ \dagger}(B_{\infty})|}{n^{ \frac{a-1/2}{a-2}}}\right)_{t \geq 0} 
 &\xrightarrow[n\to\infty]{(d)}&
  \left( \mathsf{p}_{ \mathbf{q}} \cdot S^\uparrow_{\vartheta_{t / \mathsf{h}_{ \mathbf{q}}}},  \mathsf{v}_{ \mathbf{q}} \cdot Z_{\vartheta_{t/ \mathsf{h}_{ \mathbf{q}}}}\right)_{t \geq 0}.
 \end{eqnarray}

\proof[Proof of Theorem \ref{thm:scalinglayers}.] The proof of Theorem \ref{thm:scalinglayers} again follows the steps of \cite{CLGpeeling} and we therefore only sketch the structure and highlight the main changes. We denote by $ \mathfrak{e}_{0} \subset  \mathfrak{e}_{1} \subset \cdots \subset B_{\infty}$ the peeling process of $B_{\infty}$ using the algorithm $ \mathcal{L}^\dagger$. The idea is to consider the speed at which the peeling with algorithm $ \mathcal{L}^\dagger$ ``turns'' around the boundary. 
To make this precise we introduce for $r\geq 0$ the sets $\mathcal{H}_r$ of oriented boundary edges of $\overline{ \mathrm{Ball}}_{r}^{\dagger}( B_{\infty})$ that have the unique hole on their right.
These can be naturally viewed\footnote{Notice that two oriented boundary edges in $\overline{ \mathrm{Ball}}_{r}^{\dagger}( B_{\infty})$ may appear as opposite orientations of a single edge of $B_\infty$, since in the peeling operation two boundary edges may be identified.} as sets of oriented edges in $B_{\infty}$, allowing us to define their union $\mathcal{H} = \bigcup_{r\geq 0} \mathcal{H}_r$ consisting of all oriented edges in $B_{\infty}$ that belong to the boundary of some ball $\overline{ \mathrm{Ball}}_{r}^{\dagger}( B_{\infty})$.  
We let $A_{n}$ be the number of those oriented edges in $\mathcal{H}$ that have been ``swallowed'' by $\mathfrak{e}_{n}$, i.e. that are present in $\mathfrak{e}_n$ but do not correspond to a boundary edge of $\mathfrak{e}_n$. Then we claim that
  \begin{eqnarray}   
  \label{analoProp13}
  \frac{A_{n}}{n} \xrightarrow[n\to\infty]{(P)}  \mathsf{a}_\qseq := \frac{1}{2}\left(1+ \sum_{k=0}^\infty (2k+1)\nu(k)\right).\end{eqnarray}
The idea to prove this convergence is as follows. 
First notice that at each peeling step at least one edge of $\mathcal{H}$ is swallowed, namely the peel edge itself.
To determine the remaining swallowed edges, we need some definitions.
Recall that the height of an edge $e$ incident to the hole of $ \mathfrak{e}_{i}$ is by definition $ \mathrm{d}_{ \mathrm{gr}}^\dagger(f, \rootface)$ where $f$ is the face adjacent to $e$ inside  $ \mathfrak{e}_{i}$. Let $D_{n}$  be the number of edges on the boundary of $ \mathfrak{e}_{n}$ at height $H_{n}$, the other $G_{n}:=2P_{n}-D_{n}$ edges being at height $H_n+1$. We claim that, for most times $n$ both $G_{n}$ and $D_{n}$ are large enough such that, except on a set of small probability, the number of swallowed edges of $\mathcal{H}$ (in addition to the peel edge) is $2k+1$ precisely when we swallow a bubble of perimeter $2k$ on the right of the peeling point.
Since the latter event occurs with probability asymptotic to $ \frac{1}{2} \nu(-(k+1))$ when the perimeter is large, we find for the variation $\Delta A_n:=A_n-A_{n-1}$ that 
\begin{eqnarray*}  \mathbb{E}[\Delta A_n] \approx 1 +  \frac{1}{2}\sum_{k=1}^\infty \nu(-k)(2k-1).\end{eqnarray*}
The right-hand side is easily seen to be equal to $ \mathsf{a}_{ \mathbf{q}}$ after a few manipulations using the fact that $\nu$ is centered. 

Next we claim that most of the time both $D_{n}$ and $2P_{n}-D_{n}$ are large, or more precisely that 
for every integer $L\geq 1$ we have 
 \begin{eqnarray} \label{eq:idealenvi}\frac{1}{n} \sum_{i=0}^{n} \mathbf{1}_{\{D_i\leq L\;{\it or}\;2P_{i}-D_{i}\leq L\}}
 \xrightarrow[n\to\infty]{(P)}  0.  \end{eqnarray}
To prove the last display, we first recall  from Section \ref{sec:morehtransform} that $P_{n} \to \infty$ and so $D_{n}$ and $2P_{n}-D_{n}$ cannot be both small. Next, we consider the Markov chain $(P_{n},D_{n},H_{n})$ with values in $( \mathbb{Z}_{>0}, \mathbb{Z}_{\geq 0}, \mathbb{Z}_{ \geq 0})$ whose transition kernel $Q$ is easily computed exactly (recall \eqref{eq:qp} and \eqref{eq:qp2}): for $2\leq \ell \leq 2p$ we have 
\begin{equation}
\label{transi-dual}
\begin{array}{ll}
Q((p,\ell,h),(p+k,\ell-1,h))= p_{k+1}^{(p)} &\text{for } k\geq 0\\
Q((p,\ell,h),(p-k,\ell-2k,h))= p^{(p)}_{-k+1}&\text{for }1\leq k<  \frac{\ell}{2}\\
Q((p,\ell,h),(p-k,2(p-k),h+1))= p^{(p)}_{-k+1}\qquad&\text{for }\frac{\ell}{2}\leq  k \leq p-1\\
Q((p,\ell,h),(p-k,\ell-1,h))= p^{(p)}_{-k+1}\qquad&\text{for }1\leq k <  p-(\ell-1)/2\\
Q((p,\ell,h),(p-k,2(p-k),h))= p^{(p)}_{-k+1}\qquad&\text{for }p-(\ell-1)/2\leq  k\leq p-1\,,
\end{array}
\end{equation}
while for $\ell=1$
\[Q((p,1,h),(p+k,2(p+k),h+1)) = \begin{cases} p_{k+1}^{(p)} &\text{for } k\geq 0\\2 p_{k+1}^{(p)} & \text{for } 1-p\leq k\leq -1\end{cases}.\]
Using these inputs we can adapt the proof of \cite[Lemma 12]{CLGpeeling} to obtain \eqref{eq:idealenvi}.

Given \eqref{eq:idealenvi} the proof of \eqref{analoProp13} is analogous\footnote{More precisely, the estimate on the martingale $M_{n}$ of \cite[Proposition 11]{CLGpeeling} now becomes $ \mathbb{E}[(\Delta M_{n})] \leq C n^{3-a}$ which is still sufficient for our purposes since $3-a<1$. Moreover, instead of using the rough bound $|\Delta A_{n}| \leq 1+ 2 |\Delta P_{n}|$ one should use the more precise bound $ |\Delta A_{n}| \leq 1 + 2|\Delta P_{n}|  \mathbf{1}_{ P_{n} \leq 0}$ and use \eqref{eq:bounddeltaperi}. 
} to \cite[Proposition 11 and Proposition 14]{CLGpeeling}. From here one can easily adapt \cite[convergence (54)]{CLGpeeling}, and  prove that we can combine the convergences of \eqref{eq:idealenvi} and Theorem \ref{thm:scalingperimeter+volume} to prove that jointly with the latter convergences, for any $ \varepsilon>0$ we have 
 \begin{eqnarray*} n^{- \frac{a-2}{a-1}}\left( H_{[nt]}- H_{[ \varepsilon n]}\right)_{t \geq  \varepsilon}  &\xrightarrow[n\to\infty]{(d)}&  \left(  \frac{ \mathsf{a}_{ \mathbf{q}}}{2 \mathsf{p}_{ \mathbf{a}}} \int_{ \varepsilon}^t  \frac{ \mathrm{d}u}{S^\uparrow_{u}} \right)_{t \geq  \varepsilon},  \end{eqnarray*}in distribution in the Skorokhod sense. We now let $ \varepsilon \to 0$ in the last display. This causes no problem for the right-hand side since we have seen in the proof of Proposition \ref{prop:scalingeden} that $(S^{\uparrow}_{u})^{-1}$ is almost surely integrable at $0+$. To get control over the left-hand side one must show that for any $\delta >0$ we have $ \lim_{ \varepsilon \to 0} \sup_{n \geq 1} \mathbb{P}( H_{[ \varepsilon n]}  \geq \delta n^{ \frac{a-2}{a-1}}) =0$. As in \cite[Proof of Proposition 10]{CLGpeeling}, this follow from the Markov inequality and Lemma \ref{Hexp} below, which gives control over the expectation of $H_{n}$.
\endproof

\begin{lemma}\label{Hexp}
If $a\in(2,\frac 5 2)$, then there exists a constant $C$ such that $\expec[H_n]\leq C n^{\frac{a-2}{a-1}}$ for every $n\geq 1$.
\end{lemma}
\begin{proof} We interpolate $H$ by a more ``continuous'' process and let $H_n' := H_n + \frac{G_n}{2P_n} = H_n + 1 - \frac{D_n}{2P_n}$ such that $H_{n}+1 \geq H_n' \geq H_n$ for all $n\geq 0$.
We will compute the expectation of the change $\Delta H'_n := H'_{n+1}-H'_n$ and show that there exists a $C'>0$ such that $\expec[ \Delta H'_n | \mathcal{F}_n] < C'/P_n$ for all $n$ and all $\mathcal{F}_n$.
When $(P_n,D_n,H_n)=(p,1,h)$ we easily get $\expec[\Delta H_n' | (P_n,D_n,H_n)=(p,1,h)] = \frac{1}{2p}$, so let us concentrate on the case $D_n=\ell\geq 2$.
We have
\begin{align*}
\expec[\Delta H_n' | (P_n,D_n,H_n)=(p,\ell,h)] &= \sum_{k=0}^\infty p_{k+1}^{(p)} E_0(p,\ell,k) \\
&\quad+ \sum_{k=1}^{p-1} p_{-k+1}^{(p)} (E_{\mathrm{left}}(p,\ell,-k)+E_{\mathrm{right}}(p,\ell,-k)),
\end{align*}
where the terms $E_0(p,\ell,k)$, $E_{\mathrm{left}}(p,\ell,-k)$, and $E_{\mathrm{right}}(p,\ell,-k)$ correspond to the contributions of respectively the first line, the second and third line, and the last two lines of the transition kernel \eqref{transi-dual}.
A simple calculation shows that they satisfy
\begin{align*}
E_0(p,\ell,k) &=\frac{\ell}{2p}-\frac{\ell-1}{2(p+k)}= \frac{p+k\ell}{2p(p+k)}\leq \frac{1+k}{p+k},\\
E_{\mathrm{left}}(p,\ell,-k)&=\frac{\ell}{2p}-\left(\frac{\ell-2k}{2(p-k)}\vee 0\right) \leq \frac{k}{p},\\
E_{\mathrm{right}}(p,\ell,-k)&=\frac{\ell}{2p}-\left(\frac{\ell-1}{2(p-k)}\wedge 1\right) \leq \frac{k}{p}. 
\end{align*}

Using that $\sqrt{k} \leq h^\uparrow(k) \leq 2\sqrt{k}$ for all $k \geq 0$ we then obtain the bounds
\begin{align*}
\sum_{k=0}^\infty p_{k+1}^{(p)} E_0(p,\ell,k) &\leq 2 \sum_{k=1}^\infty \frac{(k+1) \nu(k)}{\sqrt{p(p+k)}}\leq \frac{2}{p} \sum_{k=0}^\infty (k+1)\nu(k) = \frac{C_0}{p},\\
\sum_{k=1}^{p-1} p_{1-k}^{(p)} (E_{\mathrm{left}}(p,\ell,-k)&+E_{\mathrm{right}}(p,\ell,-k)) \leq \frac{1}{p} \sum_{k=1}^{p-1} \frac{h^\uparrow(p-k)}{h^\uparrow(p)} k\nu(-k) \leq \frac{1}{p} \sum_{k=1}^{\infty} k\nu(-k)=\frac{C_1}{p}.
\end{align*}
Combining these we conclude that $\expec[\Delta H_n' | (P_n,D_n,H_n)=(p,\ell,h)] \leq C'/p$ for all triples $(p,\ell,h)$ and therefore $\expec[ \Delta H'_n ] \leq C'' n^{-1/(a-1)}$ by  \eqref{eq:estim1/Pn}.
It follows that $\expec[H_n] \leq \expec[H'_n] \leq C n^{\frac{a-2}{a-1}}$ for some $C>0$. 
\end{proof}

\section{The dense phase}
\begin{center} \hrulefill  \quad \textit{In this section we suppose that $a \in \left( \frac32;2\right)$ } \hrulefill  \end{center}
We now focus on the study of the dense phase corresponding to $ a \in (3/2;2)$. We start with an easy but yet striking result in the case of the Eden model and then move to the more precise study of the geometry of  $ B_{\infty}^\dagger$. 
\subsection{Eden model and transience}
Recall that $ \mathrm{d_{fpp}}(\cdot,\cdot)$ is the first-passage percolation metric on $ {B}_{\infty}^\dagger$ for which its edges are endowed with i.i.d.\,exponential weights. As usual $ \rootface$ denotes the root face of $B_{\infty}$ which is the origin of $ B_{\infty}^\dagger$. 
\begin{proposition} \label{prop:finiteexpect} When $a \in (3/2;2)$ we have 
\[  \mathbb{E}\left[\mathrm{d_{fpp}}( \rootface, \infty)\right] = \expec[N_0] < \infty,\]
where $\mathrm{d_{fpp}}( \rootface, \infty)$ is the infimum of the \textsc{fpp}-length of all infinite paths in $ B_{\infty}^\dagger$, and $N_0$ is the number of times the random walk $(W_i)_{i \geq 0}$ started at 1 visits 0.
\end{proposition}
\proof We do the peeling process on $ B_{\infty}$ with the algorithm of Proposition \ref{prop:peelingeden=uniform} and recall the notation $(\tau_{i})_{i \geq 0}$ of Section \ref{sec:edenalgo}.  The proposition boils down to computing the expectation of $ \tau_{\infty} = \lim_{i \to \infty} \tau_{i}$. By Proposition \ref{prop:scalingeden}, conditionally on the perimeter process $(P_{i})_{i \geq 0}$ during the exploration, the increments $ \tau_{i+1}-\tau_{i}$ are independent exponential variables of mean $1/(2P_{i})$. Hence we have
\[ \mathbb{E}[\tau_{\infty}] = \sum_{i=0}^{\infty} \mathbb{E}\left[\frac{1}{2P_{i}}\right] \underset{ \mathrm{Lem.}\,\ref{lem:1/P}}{=} \sum_{k=1}^\infty \prob_1(W_k=0) = \expec[N_0].\]
From the local limit theorem \cite[Theorem 4.2.1]{IL71} we have $ \mathbb{P}_{1}(W_{k}=0) \sim C_{0}\, k^{-1/(a-1)}$ as $k \to \infty$ for some constant $C_{0}>0$ and so when $a \in (3/2;2)$ we have $ \mathbb{E}[N_{0}]< \infty$  (in other words  the walk $(W_i)_{i\geq 0}$ is transient whenever $a <2$).
\endproof

\begin{corollary} \label{cor:transient} When $a \in( 3/2;2)$ the random lattice $B^\dagger_{\infty}$ is almost surely transient (for the simple random walk).
\end{corollary}
\proof We use the method of the random path \cite[Section 2.5 page 41]{LP10}. More precisely, the \textsc{fpp} model on $B^\dagger_{\infty}$ enables us to distinguish an infinite oriented path $\vec{\Gamma} : \rootface \to \infty$ in $B_{\infty}^\dagger$ which is the shortest infinite path starting from the origin for the \textsc{fpp}-distance
 (uniqueness of this path is easy to prove). In our case, this path can equivalently be seen as an unoriented path $\Gamma$ since it is simple. From this path $\vec{\Gamma}$ one constructs a unit flow $\theta$ on the directed edges with source at $\rootface$ by putting for any oriented edge $ \vec{e}$ of $B_{\infty}^\dagger$
\[ \theta( \vec{e}) = \mathbb{P}_{ \mathrm{fpp}}( \vec{e} \in \vec{\Gamma}) - \mathbb{P}_{ \mathrm{fpp}}( \reflectbox{\ensuremath{\vec{\reflectbox{\ensuremath{e}}}}} \in \vec{\Gamma}).\]
To show that the energy of this flow is finite, we compare it to the  expected \textsc{fpp}-length of $\vec{\Gamma}$  which is almost surely finite by Proposition \ref{prop:finiteexpect}. More precisely, if $x_{e}$ denotes the exponential weight on the edge $e$, we just remark that there exists a constant\footnote{{In fact one can take $C = \inf_{s>0} \left(\int_{0}^{s} \mathrm{d}x \, x e^{-x} \right)/ \left(\int_{0}^{s} \mathrm{d}x \,e^{-x} \right)^{2}= \frac{1}{2}$}} $C>0$ such that for any event $A$ we have 
\[ \mathbb{E}_{ \mathrm{fpp}}\left[ x_{e} \mathbf{1}_{A}\right] \geq C\, \mathbb{P}_{ \mathrm{fpp}}(A)^2.\] Indeed, if $\delta = \mathbb{P}(A)$ we have
$\mathbb{E}_{ \mathrm{fpp}}[ x_{e} \mathbf{1}_{A}] \geq \mathbb{E}_{ \mathrm{fpp}}[ x_{e} \mathbf{1}_{A} \mathbf{1}_{x_{e} \geq \delta/2}] \geq \delta/2\, \mathbb{P}_{ \mathrm{fpp}}(A \cap \{ x_{e} \geq \frac{\delta}{2}\})$ and use the fact that $\mathbb{P}_{ \mathrm{fpp}}(A \cap \{ x_{e} \geq \frac{\delta}{2}\}) \geq \mathbb{P}_{ \mathrm{fpp}}(A)+ \mathbb{P}_{ \mathrm{fpp}}(x_{e} \geq \frac{\delta}{2}) -1 = \delta + e^{-\delta/2}-1 \geq \delta/2$. Using this we can write
\begin{align*} \sum_{\vec{e} \in \overrightarrow{\mathsf{Edges}}( B_{\infty})} \theta(\vec{e})^2 &\leq  4 \sum_{e \in\mathsf{Edges}( B_{\infty}) } \mathbb{P}_{ \mathrm{fpp}}( e \in \Gamma)^2   \leq \frac{4}{C}
 \sum_{e \in\mathsf{Edges}( B_{\infty}) } \mathbb{E}_{ \mathrm{fpp}}\left[ \mathbf{1}_{e \in \Gamma} x_{e}\right] \\&=  \frac{4}{C}\, \mathbb{E}_{ \mathrm{fpp}}[ \mathrm{Length_{fpp}}(\Gamma)] <\infty. \end{align*}
This proves almost sure transience of the lattice as desired.  \endproof

\subsection{Dual graph distance}
We now come back to the dual graph distance $ \mathrm{d}_{ \mathrm{gr}}^\dagger$ on $B_{\infty}^\dagger$.   Our main result which parallels Theorem \ref{thm:scalinglayers} is the following:

\begin{theorem} \label{thm:dense} For $a \in (3/2;2)$ there exists a constant $ \mathsf{c}_{a} \in (0,\infty)$ such that with the same notation as in the geometric interpretation below Theorem \ref{thm:scalinglayers} we have the following convergences in probability
\[ r^{-1}\log\left(\left|\partial \mathrm{Ball}^\dagger_r(B_\infty)\right|\right) \xrightarrow[r\to\infty]{ (\mathbb{P})} \mathsf{c}_{a}, \qquad  r^{-1}\log\left(\left|\mathrm{Ball}^\dagger_r(B_\infty)\right|\right) \xrightarrow[r\to\infty]{ (\mathbb{P})}  ({a-1/2}) \cdot \mathsf{c}_{a}.\]
\end{theorem}

The proof of the above theorem is presented in the next section. It mainly relies on Proposition \ref{prop:cvZ} which enables us to see, in the scaling limit, the different times needed for the algorithm $ \mathcal{L}^\dagger$ to complete a full layer, whereas in the dilute phase this information vanishes in the scaling limit. In order to make the proof more digestible, we postpone a few technical estimates to Section \ref{sec:technics}

\subsubsection{Scaling limit of the peeling with algorithm $ \mathcal{L}^\dagger$ in the dense phase}

We perform the peeling process on $ B_{\infty}$ with algorithm $ \mathcal{L}^\dagger$ of Section \ref{sec:algodual}. Recall that $\theta_{r}$ is the first time $i$ when all the faces adjacent to the unique hole of $ \mathfrak{e}_{i}$ are at dual distance at least $r$ from the root face $\rootface$ of $B_{\infty}$.
   
 We shall need to generalize a bit the setup such that during the peeling with algorithm $ \mathcal{L}^\dagger$, we start at time $0$ with a boundary of length $2p$ with $p \geq 1$ (or equivalently that the root face of $B_{\infty}$ has degree $2p$) while still denoting by $\theta_{1}, \theta_{2}, \ldots$ the times it takes to complete one layer, two layers etc. We denote by $ \mathbb{P}_{p}$ and $ \mathbb{E}_{p}$ the corresponding probability and expectation. By the Markov property of the exploration of $ B_{\infty}$ we know that the law of $P_{\theta_{r+1}}$ under $ \mathbb{P}_{1}$ conditionally on $P_{\theta_{r}}=p$ is that of 
$P_{\theta_{1}}$ under $ \mathbb{P}_{p}$. Recall also from Section \ref{sec:dilutelayers} that $D_{i}$ denotes the number of edges on the boundary at minimal height $H_{i}$ after $i$ peeling steps. 
 We now introduce the scaling limit of $(P_{i})_{i\geq0}$,$(D_{i})_{i\geq 0}$ and $(\theta_{i})_{i \geq 0}$ under $ \mathbb{P}_{p}$ when $p \to \infty$.

We first consider $(S^\uparrow_{t})_{ t \geq 0}$ the $(a-1)$-stable Lévy process conditioned to stay non-negative with positivity parameter $\rho$ satisfying $ (1-\rho)(a-1)= \frac{1}{2}$ already introduced in Section \ref{sec:perimeter} but now started from $S_{0}^\uparrow =1$. 
By an  extension of Proposition \ref{prop:scalingperimeter}  (which is granted by \cite{CC08}), we know that $S^{\uparrow}$ is the scaling limit of the perimeter process $P$ under $ \mathbb{P}_{p}$ as $p \to \infty$ in the sense that under $ \mathbb{P}_{p}$
 \begin{eqnarray}  \label{eq:extension1} 
\left(\frac{P_{[t(p/\mathsf{p}_{ \mathbf{q}})^{a-1}]}}{p}\right)_{t \geq 0} & \xrightarrow[p\to\infty]{(d)}&   (S^\uparrow_{t})_{ t \geq 0}, \end{eqnarray} in distribution in the Skorokhod sense as $p \to \infty$.  We now introduce the scaling limit of $D$ by mimicking in the continuous setting the behavior of $D$ with respect to $P$.
In the case when $a <2$, the process $S^\uparrow$ is  pure jump and we can write
\[ S_{t}^\uparrow = 1+\sum_{t_{i} \leq t} \Delta S^\uparrow_{t_{i}},\]
where $t_{1},t_{2}, \ldots$ is a measurable enumeration of its jumps times and $\Delta S^\uparrow_{t} = S^\uparrow_{t}-S^\uparrow_{t^-}$. Independently of $(S^\uparrow_t)_{t\geq 0}$ let also $( \epsilon_{i})_{ i \geq 1}$ be independent fair coin flips taking values in $\{ \mathrm{right}, \mathrm{left}\}$. With these ingredients  we build a new pure jump process $( \mathcal{D}_{t})_{t \geq 0}$ by putting $ \mathcal{D}_{0}= 1$ and for every jump time $t_{i}$ such that $\Delta S^\uparrow_{t_{i}} <0$ is a negative jump we put
 \begin{eqnarray} \label{eq:rules} \Delta  \mathcal{D}_{t_{i}} = \left\{ \begin{array}{ll}   \Delta S^\uparrow_{t_{i}} & \mbox{ if }  \epsilon_{i} = \mathrm{right}\\
\min\big(0,(S^\uparrow_{t_{i}^-}- \mathcal{D}_{t_{i}^-})+\Delta S^\uparrow_{t_{i}}\big) & \mbox{ if } \epsilon_{i} = \mathrm{left},  	\end{array} \right.  \end{eqnarray} as long as $ \mathcal{D}$ stays positive. More precisely, with the above construction, the process $ \mathcal{D}$ is pure jump and (a.s. strictly) non-increasing; we let $\zeta_{1}= \inf \{ t \geq 0 :  \mathcal{D}_{t} < 0\}$ and at time $\zeta_{1}$ we change the value of $ \mathcal{D}_{\zeta_{1}}$ (which otherwise would be strictly negative) and set its new value to be 
\[ \mathcal{D}_{\zeta_{1}} := S_{\zeta_{1}}^\uparrow.\]
From this time on, we apply the rules of \eqref{eq:rules} until $ \mathcal{D}_{t}$ reaches a strictly negative value a time $\zeta_{2}$. Then we reset $ \mathcal{D}_{\zeta_{2}} := S^\uparrow_{\zeta_{2}}$ and iterate the above procedure to construct the full process $ (\mathcal{D}_{t})_{t \geq 0}$ and the sequence of random times $(\zeta_{i})_{i \geq 1}$. See Fig.\,\ref{fig:constructionD}. As promised, these processes are the scaling limits of the discrete processes $(P,D,\theta)$ in the following sense:

\begin{figure}[t]
 \begin{center}
 \includegraphics[width=9cm]{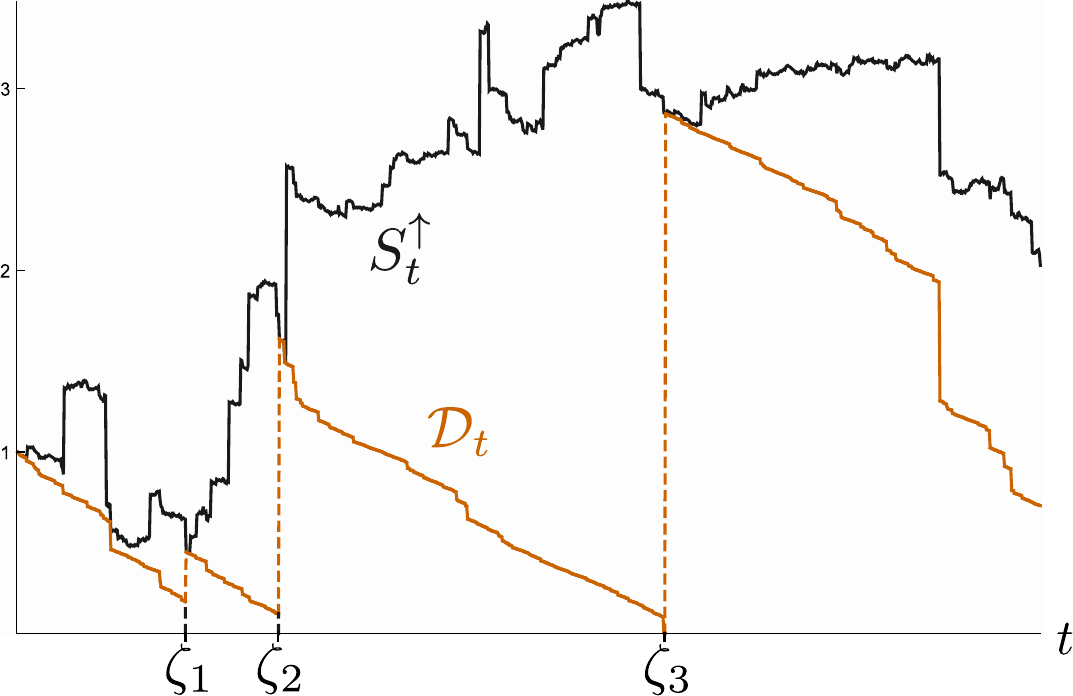}
 \caption{ \label{fig:constructionD}Illustration of the construction of the process $ \mathcal{D}$ from the process $S^\uparrow$ and a sequence of independent coin flips.}
 \end{center}
 \end{figure}

\begin{proposition} \label{prop:cvZ}
We have the following convergences in distribution under $ \mathbb{P}_{p}$
 \begin{equation} \left(\left(  \frac{P_{[t(p/\mathsf{p}_{ \mathbf{q}})^{a-1}]}}{p},\frac{D_{[t(p/\mathsf{p}_{ \mathbf{q}})^{a-1}]}}{2p}\right)_{{t \geq 0}}\!\!,\left(\frac{\theta_{i}}{(p/\mathsf{p}_{ \mathbf{q}})^{a-1}}\right)_{i \geq 1}\right) \!\xrightarrow[p\to\infty]{(d)}\!  \left(\left( S^\uparrow_{t}, \mathcal{D}_{t} \right)_{ t \geq 0}\!\!,(\zeta_{i})_{i \geq 1}\right)\label{eq:convDP}  \end{equation}
 furthermore, jointly with the above convergences we have $(\frac{P_{\theta_{i}}}{p})_{i \geq 1} \to (S^\uparrow_{\zeta_{i}})_{i \geq 1}$ in law.  \end{proposition}

\begin{remark} Let us explain heuristically a crucial difference between the dilute phase $a \in (2;5/2)$ and the dense phase $ a \in (3/2 ; 2)$ above. In the dilute phase, by \eqref{analoProp13} the time needed for the peeling process with algorithm $ \mathcal{L}^\dagger$ to ``turn around'' a boundary of length $p$ and discover a new layer is roughly of order $p$  whereas the scaling in time for the process $(P)$ is $p^{a-1}$ which is much larger than $p$. So the information given by the $(\theta_{i})_{i \geq 1}$ disappears in the scaling limit. In the dense phase however, the time needed to turn around a boundary of perimeter $p$ is roughly $p^{a-1}$ which is precisely the time scaling for the process $(P)$.
\end{remark}

\proof[Proof of Proposition \ref{prop:cvZ}] {The convergence of the rescaled process $P$ towards $S^{\uparrow}$  is given in \eqref{eq:extension1}}. Next, it is easy to see that the definition of $\mathcal{D}$ mimics the discrete evolution of $D$. More formally, for   $i \geq 0$ such that   $\Delta P_{i} = P_{i+1}-P_{i} < 0$ we can define $ \tilde{\epsilon}_{i} \in \{ \mathrm{left}, \mathrm{right}\}$ indicating whether the peeling process swallows a bubble on the left or on the right-hand side of the peeling point. By the probability transitions of the peeling process, conditionally on $(P)$ these variables are independent and uniformly distributed over the two outcomes. We put $\tilde{\epsilon}_{i}= \mathrm{center}$ when the peeling process discovers a new face i.e.~when $\Delta P_{i} \geq 0$. Then using \eqref{transi-dual} we see that for $0 \leq i < \theta_{1}-1$ we have 
 \begin{eqnarray}\label{eq:defimimicD} \Delta D_{i} = \left\{ \begin{array}{ll} 2 \Delta P_{i} 
 & \mbox{ if }\tilde{\epsilon}_{i} = \mathrm{right} \\
 \min(0,(2P_{i}-D_{i})+2 \Delta P_{i}+1) - 1 & \mbox{ if }  \tilde{\epsilon}_{i} = \mathrm{left}\\
-1& \mbox{ if }  \tilde{\epsilon}_{i} = \mathrm{center}\\
 	\end{array} \right.  \end{eqnarray} At time $\theta_{1}$ we then have $D_{\theta_{1}} = 2 P_{\theta_{1}}$ and iterate the last construction for times $\theta_{1} \leq i < \theta_{2}-1$ etc. The above construction of $D$ is the discrete analog of the continuous construction of $ \mathcal{D}$ given in \eqref{eq:rules},  the various factors of two which differ between the above display and \eqref{eq:rules} come from the fact that $P_{i}$ is the half-perimeter at time $i$ whereas $D_{i}$ counts the number of edges at height $H_{i}$ (not divided by two). By the Markov property and the similarity of the constructions of $(D,P)$ and $( \mathcal{D}, S^{\uparrow})$ it is sufficient to prove the convergence until the completion of one layer, that is jointly with \eqref{eq:extension1} we have
	 \begin{equation} \left(\left( \frac{D_{[t(p/\mathsf{p}_{ \mathbf{q}})^{a-1}]}}{2p}\right)_{t \in [0, \theta_{1}/(p/\mathsf{p}_{ \mathbf{q}})^{a-1}]}, \frac{\theta_{1}}{(p/\mathsf{p}_{ \mathbf{q}})^{a-1}}, \frac{P_{\theta_{1}}}{p} \right) \xrightarrow[p\to\infty]{(d)}  \left(\left( \mathcal{D}_{t} \right)_{ t \in [0, \zeta_{1})},\zeta_{1}, S^{\uparrow}_{\zeta_{1}}\right).\label{eq:convDPbis}  \end{equation}
To prove the above display it is convenient to argue by approximation. Fix $ \varepsilon>0$ and denote by $ \mathcal{D}^{( \varepsilon)}$ the process obtained by repeating the construction of $ \mathcal{D}$ from $ S^{\uparrow}$ but only keeping those (negative) jumps of $ S^{\uparrow}$ of absolute size at least $  \varepsilon$. We define accordingly $ \zeta_{1}^{( \varepsilon)}$ to be the first time at which $ \mathcal{D}_{t}^{( \varepsilon)}$ becomes strictly negative. We do the same approximation procedure in the discrete setting and define a process $D^{( \varepsilon)}$ starting from $p$ by applying the rules \eqref{eq:defimimicD} restricted to (negative) jumps of $P$ of size at least $ \varepsilon p$ (in particular the third line in \eqref{eq:defimimicD} is never used) and also define $\theta_{1}^{( \varepsilon)}$ as the first time the process $D^{( \varepsilon)}$ reaches a negative value. Notice that there are only finitely many (random) times before $\theta_{1}$ (resp.~$\zeta_{1}$) for which $P$ (resp.~$S^{ \uparrow}$) has a negative jump of absolute size larger than $  \varepsilon p$ (resp.~$ \varepsilon$) and that for fixed $ \varepsilon$ the process $S^{\uparrow}$ has no jump of size exactly $ \varepsilon$. These facts combined with the convergence in distribution in the Skorokhod sense \eqref{eq:extension1} and with the fact that the variables $ \epsilon_{i}$ and $ \tilde{ \epsilon}_{i}$ are i.i.d.~and uniform over $\{  \mathrm{left}, \mathrm{right}\}$ entails that jointly with \eqref{eq:extension1} we have
\begin{equation} \left(\!\left(\! \frac{D^{( \varepsilon)}_{[t(p/\mathsf{p}_{ \mathbf{q}})^{a-1}]}}{2p}\right)_{t \in [0, \theta^{( \varepsilon)}_{1}/(p/\mathsf{p}_{ \mathbf{q}})^{a-1}]}\!\!\!\!\!\!\!, \frac{\theta^{(\varepsilon)}_{1}}{(p/\mathsf{p}_{ \mathbf{q}})^{a-1}}, \frac{P_{\theta^{(\varepsilon)}_{1}}}{p} \right)\!\xrightarrow[p\to\infty]{(d)}\! \left(\left( \mathcal{D}^{(\varepsilon)}_{t} \right)_{ t \in [0, \zeta^{(\varepsilon)}_{1})}\!\!,\zeta^{(\varepsilon)}_{1}, S^{\uparrow}_{\zeta^{(\varepsilon)}_{1}}\right).\label{eq:convDPbis2}  \end{equation}
We wish to let $ \varepsilon \to 0$ but for this we need some uniform control with respect to $p$ for the left-hand side. We begin with the right-hand side: since $a-1 <1$ we know that  $ S^{\uparrow}$ is pure jump and so for any given $t >0$ we have
 \begin{eqnarray} \sum_{t_{i} \leq t} |\Delta S^{\uparrow}_{t_{i}}| \mathbf{1}_{ |\Delta S^{\uparrow}_{t_{i}}|> \varepsilon} \xrightarrow[ \varepsilon \to 0]{a.s.} 0.  \label{eq:petitsauts} \end{eqnarray}
It follows from the last display and the definitions of $ \mathcal{D}^{ (\varepsilon)} $ and $ \mathcal{D}$ that we have the following almost sure convergences in the sense of Skorokhod 
 \begin{eqnarray} ( \mathcal{D}^{( \varepsilon)}_{t})_{ t \in [0, \zeta_{1}^{( \varepsilon)})} \xrightarrow[ \varepsilon\to 0]{a.s.} ( \mathcal{D}_{t})_{t \in [0, \zeta_{1})}, \quad \zeta_{1}^{( \varepsilon)}  \xrightarrow[ \varepsilon\to 0]{a.s.} \zeta_{1}, \quad S^{\uparrow}_{\zeta_{1}^{( \varepsilon)}} \xrightarrow[ \varepsilon\to 0]{a.s.} S^{ \uparrow}_{\zeta_{1}}.   \label{eq:convcontinuous}\end{eqnarray}
Similarly, in the discrete setting we can use \eqref{eq:bounddeltaperi} to get that for any $\delta,t \geq 0$ we have
\[ \sup_{p \geq 1}  \mathbb{P}_{p}\left( \sum_{i=0}^{t \cdot p^{a-1} } \mathbf{1}_{|\Delta P_{i}|< \varepsilon p} |\Delta P_{i}|  >  \delta\, p\right) \xrightarrow[ \varepsilon \to 0]{} 0.\]
Using the fact that  $|\Delta D_{i}| \leq 1 + 2 |\Delta P_{i}| \mathbf{1}_{ \Delta P_{i} \leq 0}$ we consequently have 
\[ \sup_{p \geq 1}  \mathbb{P}_{p}\left( \sum_{i=0}^{t \cdot p^{a-1}} \mathbf{1}_{|\Delta D_{i}|< \varepsilon p} |\Delta D_{i}|  > \delta\, p\right) \xrightarrow[ \varepsilon \to 0]{} 0.\] 
It is then standard to combine the last two displays and the properties of $D$ and $ D^{( \varepsilon)}$ to deduce that for any $t \geq 0$, if $ \|\cdot\|$ denotes the Skorokhod distance between two functions over the time interval $[0,t)$ then we have 
  \begin{eqnarray} \label{eq:uniformp} \sup_{p \geq 1} \mathbb{P}_{p}\left(p^{-1}\left\|D^{( \varepsilon)}_{\cdot\, (p/ \mathsf{p}_{q})^{a-1}}-D_{\cdot\, (p/ \mathsf{p}_{q})^{a-1}} \right\| > \delta \right) \xrightarrow[ \varepsilon \to0]{} 0.  \end{eqnarray}
Now, combining \eqref{eq:uniformp}, \eqref{eq:convcontinuous} and \eqref{eq:convDPbis2} we can deduce the convergence in law of the first components in \eqref{eq:convDPbis}. The other joint convergences in law are derived similarly. We leave the details to the reader. \endproof
	
We now introduce the following key random variable
\[ \mathcal{Z} = \log ( S^\uparrow_{\zeta_{1}}).\]
\begin{lemma} \label{lem:EspZ>0} The expectation of $ \mathcal{Z}$ denoted by $ \mathsf{c}_{a}$ is (stricly) positive.
\end{lemma}
\proof By the Markov property and the scale invariance property of the process $ (S^\uparrow)$ used in the construction of $ \mathcal{D}$ it is easy to see that conditionally on the past information up to time $\zeta_{k}$ we have 
 \begin{eqnarray} \label{eq:markovfort}(\zeta_{k+1}-\zeta_{k}, S^{\uparrow}_{\zeta_{k+1}})  & \overset{(d)}{=} & ( (S^{\uparrow}_{\zeta_{k}})^{1/(a-1)}\cdot \tilde{\zeta}_{1}, S^{\uparrow}_{\zeta_{k}} \cdot \tilde{S}^{\uparrow}_{\zeta_{1}}), \end{eqnarray} where the process $(\tilde{\zeta}, \tilde{ \mathcal{D}},\tilde{S}^{\uparrow})$ is an independent copy of $(\zeta, \mathcal{D}, S^{\uparrow})$. In particular for any $k \geq 1$ the random variable $ \log (S_{\zeta_{k}}^\uparrow)$ is obtained by summing $k$ independent copies of the variable $  \log(S_{\zeta_{1}}^\uparrow)$. Hence we have 
 \begin{eqnarray} \label{eq:egalite1/k} \mathbb{E}[ \mathcal{Z}]=\mathbb{E}[\log(S_{\zeta_{1}}^\uparrow)] = \frac{1}{k} \mathbb{E}[\log(S_{\zeta_{k}}^\uparrow)].  \end{eqnarray}
Now, when $k \to \infty$, using \eqref{eq:markovfort} and the fact that $S^{\uparrow}$ remains positive, it is any easy matter to see that $\zeta_{k} \to \infty$ hence $S_{\zeta_{k}}^\uparrow \to \infty$ and $ \log(S_{\zeta_{k}}^\uparrow) \to \infty$ as well. On the other hand, $\log(S_{\zeta_{k}}^\uparrow)$ is obviously bounded from below by the logarithm of the overall infimum $ \underline{S}^\uparrow_{\infty} = \inf \{ S_{t}^\uparrow : t \geq 0\}$ of the process $ (S^\uparrow_{t})_{t \geq0}$. Since $S^\uparrow$ is the $h$-transform of the process $S$ for the function $h(x)= \sqrt{x}$ it follows that  for any $ \varepsilon >0$ if $ T_{ \varepsilon}(X) = \inf \{ t \geq 0 : X_{t} \leq \varepsilon\}$ then we have 
 \begin{eqnarray} \label{eq:tailoverallinfimum} \mathbb{P}( \underline{S}^\uparrow_{\infty} \leq \varepsilon) = \mathbb{P}( T_{ \varepsilon}( S^\uparrow) < \infty) = \mathbb{E}[ h( S_{T_{ \varepsilon}}) \mathbf{1}_{ T_{ \varepsilon}(S) < \infty} \mathbf{1}_{ S_{t} \geq 0, \forall 0 \leq t \leq T_{ \varepsilon}(S)}] \leq \sqrt{ \varepsilon},  \end{eqnarray} from which one deduces that $\log(\underline{S}_{\infty}^\uparrow)$ is integrable. Using all these ingredients we can apply  Fatou's lemma  and get \[\liminf_{k \to \infty} \mathbb{E}[\log(S_{\zeta_{k}}^\uparrow)] \geq  \mathbb{E}\left[ \liminf_{k \to \infty} \log(S_{\zeta_{k}}^\uparrow)\right] = \infty.\]
It follows from the last display and \eqref{eq:egalite1/k} that for some $k_{0} \geq 1$ we have $  \mathbb{E}[ \mathcal{Z}]=\mathbb{E}[\log(S_{\zeta_{k_{0}}}^\uparrow)]/k_{0}>0$ as wanted. (Notice that at this point it could be that $ \mathsf{c}_{a}= \infty$ but this will be ruled out in the next proof). \endproof 

\proof[Proof of Theorem \ref{thm:dense}] By Proposition \ref{prop:cvZ} we have the convergence  $\log(p^{-1} P_{\theta_{1}}) \to \mathcal{Z}$ in distribution under $ \mathbb{P}_{p}$ as $p \to \infty$ and on the other hand  Lemma \ref{lem:estimtech} implies that the laws of $\log(P_{\theta_{1}}/p)$ under $ \mathbb{P}_{p}$ are uniformly integrable for $p\geq 1$. It follows that  \begin{eqnarray} \label{eq:cvversint} \mathbb{E}_{p}\left[ \log\left(\frac{P_{\theta_{1}}}{p}\right)\right]  \xrightarrow[p\to\infty]{} \mathbb{E}[ \mathcal{Z}] = \mathsf{c}_{a},  \end{eqnarray} and in the same time we deduce that $ \mathsf{c}_{a}$ is finite (and positive thanks to Lemma \ref{lem:EspZ>0}). We are now in position to prove a law of large numbers for  $\log( P_{\theta_{r}})$ under $\mathbb{P}_{1}$. Denote $ (\mathcal{F}_{n})_{n \geq 0}$ the filtration generated by the peeling exploration and recall that the law of $P_{\theta_{r+1}}$ under $ \mathbb{P}_{1}( \cdot \mid  \mathcal{F}_{\theta_{r}})$ is that of $\tilde{P}_{\tilde\theta_{1}}$ under $ \tilde{\mathbb{P}}_{\theta_{r}}$ where the $\sim$ means that this is a new sampling of the process. For $r \geq 1$ large we evaluate
 \begin{eqnarray} \label{eq:lgn1} \mathbb{E}_{1} \left[ \left( \log( P_{\theta_{r}}/P_{\theta_{0}}) - r \mathsf{c}_{a}\right)^2 \right] = \sum_{ 1 \leq i,j \leq r} \mathbb{E}\left[ \left(\log\left( \frac{P_{\theta_{i}}}{P_{\theta_{i-1}}}\right) - \mathsf{c}_{a}	\right)\left(\log\left( \frac{P_{\theta_{j}}}{P_{\theta_{j-1}}}\right) - \mathsf{c}_{a}	\right)	 \right].  \end{eqnarray}
The terms where $i=j$ are bounded above by some constant according to Lemma \ref{lem:estimtech}. For the other terms when $i <j$ we condition on $ \mathcal{F}_{\theta_{j-1}}$ and use the above remark to get that 
\begin{align*}\mathbb{E}\Bigg[ \left(\log\left( \frac{P_{\theta_{i}}}{P_{\theta_{i-1}}}\right) - \mathsf{c}_{a}	\right)&\left(\log\left( \frac{P_{\theta_{j}}}{P_{\theta_{j-1}}}\right) - \mathsf{c}_{a}	\right)	 \Bigg] \\&= \mathbb{E}\left[ \left(\log\left( \frac{P_{\theta_{i}}}{P_{\theta_{i-1}}}\right) - \mathsf{c}_{a}	\right) \tilde{\mathbb{E}}_{P_{\theta_{j-1}}}\left[\log\left( \frac{\tilde{P}_{\tilde{\theta}_{1}}}{P_{\theta_{j-1}}}\right) - \mathsf{c}_{a}	\right]	 \right],
\end{align*} where independently of the previous exploration, under $\tilde{\mathbb{P}}_{p}$ the random variable $\tilde{P}_{\tilde{\theta}_{1}}$ is distributed as $P_{\theta_{1}}$ under $ \mathbb{P}_{p}$.
Since we have $P_{\theta_{j}} \to \infty$ by the transience of the process $(P)$ it follows from \eqref{eq:cvversint} that \[ \tilde{\mathbb{E}}_{P_{\theta_{j-1}}}\left[\log\left( \tilde{P}_{\tilde{\theta}_1}/P_{\theta_{j-1}}\right) - \mathsf{c}_{a}	\right] \xrightarrow[j\to\infty]{a.s.} 0.\] Conditioning with respect to $ \mathcal{F}_{\theta_{i}-1}$ and using one more time the uniform integrability of the variables $\log(P_{\theta_{1}}/p)$ under $ \mathbb{P}_{p}$ we deduce that the off-diagonal terms in \eqref{eq:lgn1} go to $0$ as $i,j \to \infty$. Consequently by Cesaro's summation we have 

\[ \mathbb{E}_{1} \left[ \left( \log( P_{\theta_{r}}/P_{\theta_{0}}) - r \mathsf{c}_{a}\right)^2 \right] = o(r^2) \quad \mbox{ as }r \to \infty.\]
By Markov's inequality, this proves that $r^{-1} \log(P_{\theta_{r}}) \to \mathsf{c}_{a}$ in probability as desired in Theorem \ref{thm:dense}. The second point of the theorem follows from the first point and Lemma \ref{lem:logbehavior} below. \endproof

Recall that  the perimeter $\partial | \overline{\mathrm{Ball}}_{r}(B_{\infty})|$ is defined is terms of number of (dual) edges. It may thus be that the perimeter in terms of number of vertices on the boundary of $B_{\infty}^\dagger$ is much smaller. Whereas they are both of the same order in the dilute case (but we do not prove it), this is far from being true in the dense case since for $a \in (3/2;2)$ the random map $B_{\infty}^\dagger$ contains infinitely many cut vertices separating the origin from infinity almost surely.

\proof[{Sketch of proof}.] We will show that when doing the peeling process with algorithm $ \mathcal{L}^\dagger$, then independently of the past exploration, there is a positive probability bounded away from $0$ that within the next two consecutive layers of $B_{\infty}^\dagger$ we create a cut point (i.e.~a face of $B_{\infty}$ which is folded on itself and separates the origin from infinity in $B_{\infty}$). This proves that indeed there are infinitely many cut-points in $B_{\infty}^\dagger$. Fix $r \geq 0$ and assume that $P_{\theta_{r}}=p$. We claim that with a probability which is bounded from below independently of $p$
\begin{itemize}
\item during the construction of the $(r+1)$th layer a face $\mathbf{f}$ of degree of order $p$ is created  which contributes to a fraction say at least $1/3$ of $P_{\theta_{r+1}}$,
\item during the construction of the $(r+2)$th layer, two edges of $ \mathbf{f}$ are identified in such a way that the origin and infinity are separated in $B_{\infty}^\dagger$ by $ \mathbf{f}$, thereby creating the desired cut point.
\end{itemize}
\begin{figure}[!h]
 \begin{center}
 \includegraphics[width=10cm]{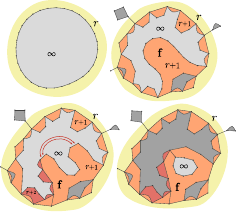}
 \caption{ \label{fig:cutpoint1}Creation of a cut point during the construction of two consecutive layers.}
 \end{center}
 \end{figure}

We leave it to the reader to translate the above recipe in terms of the process $P$ and $D$ and to use the above scaling limit given by Proposition \ref{prop:cvZ} to see that such a scenario indeed has a positive probability to happen independently of $p$. We refer to Fig.\,\ref{fig:cutpoint1} and Fig.\,\ref{fig:cutpoint2} for a pictorial description. 
 \endproof 
\begin{figure}[!h]
 \begin{center}
 \includegraphics[width=10cm]{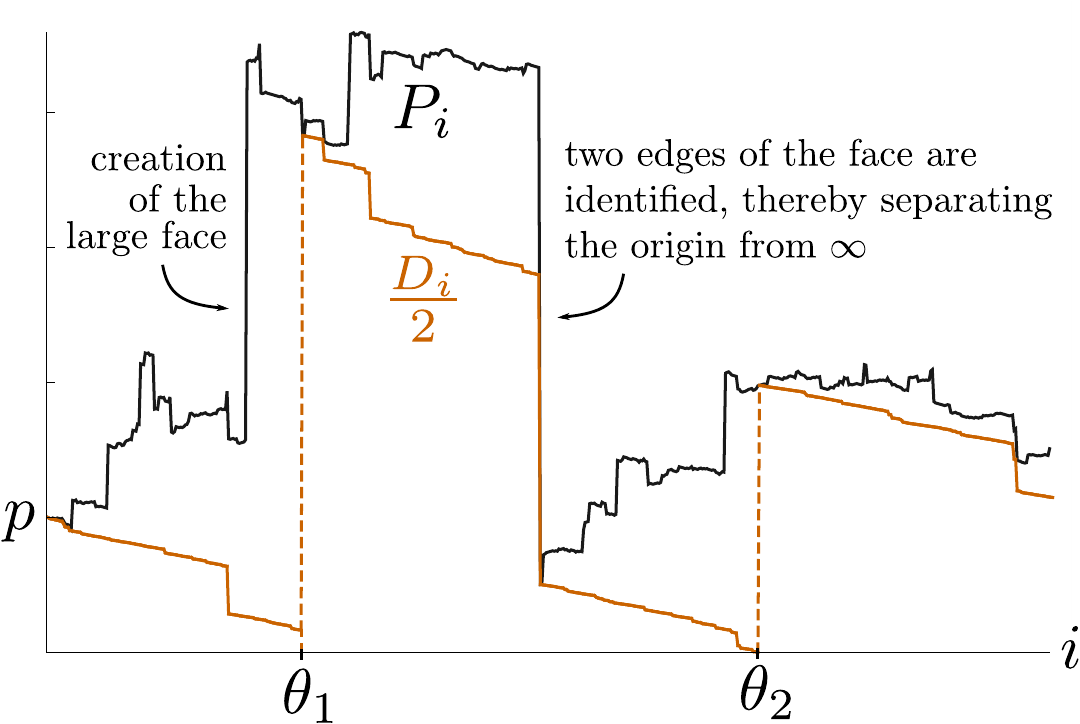}
 \caption{ \label{fig:cutpoint2} Transcription of the event of Fig.\,\ref{fig:cutpoint1} in terms of the coding processes $P$ and $D$. It is implicitly assumed that the event corresponding to the big negative jump identifies two edges that are incident to the face created in the big positive jump event.}
 \end{center}
 \end{figure}

\subsubsection{Proof of the technical estimates}
\label{sec:technics}

\begin{lemma} \label{lem:estimtech} We have
\[\sup_{p \geq 1} \mathbb{E}_{p}\left[\log^2\left( \frac{P_{\theta_{1}}}{p} \right)\right] < \infty.\]
\end{lemma}
\proof We first claim that the tail of $\theta_{1}$ under $ \mathbb{P}_{p}$ is exponential in the scale $p^{a-1}$, in other words
 \begin{eqnarray} \label{eq:expotailtau} \mathbb{P}_{p}(\theta_{1} \geq k \lfloor p^{a-1} \rfloor) \leq e^{-ck},  \end{eqnarray} for all $k \geq 1$ for some constant $c>0$ independent of $p$. The reason is the following. Suppose that $\theta_{1} \geq k p^{a-1}$, then we claim that during the time interval $[k \lfloor p^{a-1} \rfloor, (k+1) \lfloor p^{a-1} \rfloor]$  the process $(P)$ has a positive probability (independent of $p$ and $k$) to make a negative jump of size at least $p$ during which the peeling by layer process swallows at least $2p$ edges on its left. When doing so, one must necessarily complete the first layer since there are less than $2p$ edges initially at height $0$ to discover (and this number can only decrease). This easily implies (\ref{eq:expotailtau}). 
 
 To see that within a time interval of length $p^{a-1}$ the process $P$ can indeed produce a negative jump of size at least $p$ with a probability bounded away from $0$ we proceed as follows: we first produce a positive jump of size about $2p$ followed within the time interval by a negative jump of size larger than $p$. Using the explicit probability transitions for the process $P$ it is easy to see that the probability of this event is bounded away from $0$ uniformly in $p$ and in $P_{0}$. 
 
 Once we have \eqref{eq:expotailtau} in hand we first write
  \begin{eqnarray} \label{eq:splitinfsup}\mathbb{E}_{p}\left[\log^2\left( \frac{P_{\theta_{1}}}{p}\right) \right] \leq \mathbb{E}_{p}\left[\log^2\left( \frac{\underline{P}_{\theta_{1}}}{p}\right) \right] + \mathbb{E}_{p}\left[\log^2\left( \frac{\overline{P}_{\theta_{1}}}{p} \right)\right], \end{eqnarray} where $\underline{P}_{k} = \inf\{ P_{i} : 0 \leq i \leq k\}$ and $\overline{P}_{k}= \sup\{ P_{i} : 0 \leq i \leq k\}$ are the corresponding running infimum and running supremum of the process $P$. We easily take care of the first term, since $\underline{P}_{\theta_{1}}$ is bounded from below by $\underline{P}_{\infty}$ the overall infimum of $P$: a calculation similar to that of \eqref{eq:tailoverallinfimum} shows that for any  $1 \leq p' \leq p$ we have
  \[ \mathbb{P}_{p}( \underline{P}_{\infty} \leq  p') \leq C \sqrt{\frac{p'}{p}},\] for a constant $C>0$ independent of $p$ and $p'$. It follows from this  that \[ \sup_{p \geq 1} \mathbb{E}_{p}[\log^2( \underline{P}_{\infty}/p) ] < \infty\] and so $ \sup_{p \geq 1} \mathbb{E}_{p}[\log^2( \underline{P}_{\theta_{1}}/p) ] < \infty$. Let us move to the second term on the right-hand side of \eqref{eq:splitinfsup}. By splitting according to the values of $ \theta_{1}$ and applying Cauchy-Schwarz inequality we have  \begin{eqnarray} \mathbb{E}_{p}[\log^2( \overline{P}_{\theta_{1}}/p) ]&\leq& \sum_{ k \geq 1} \mathbb{E}_{p}\left[\log^2\left( \overline{P}_{k \lfloor p^{a-1} \rfloor}/p\right) \mathbf{1}_{\theta_{1} \in  [(k-1) \lfloor p^{a-1} \rfloor, k \lfloor p^{a-1} \rfloor)}\right] \nonumber \\ &\leq& 
 \sum_{ k \geq 1} \sqrt{\mathbb{E}_{p}[\log^4( \overline{P}_{k \lfloor p^{a-1} \rfloor}/p)]\cdot  \mathbb{P}_{p}(\theta_{1} \geq (k-1) \lfloor p^{a-1} \rfloor)}.   \label{eq:apresCS}  \end{eqnarray}
We will show below the rough estimate   \begin{eqnarray} \label{eq:rough} 
\mathbb{E}_{p}[\log^4( \overline{P}_{k \lfloor p^{a-1} \rfloor}/p)] \leq C'\, k  \end{eqnarray}
  for some $C'>0$ (independent of $k$ and $p$ but which may depend on $a \in (3/2;2)$) which combined with \eqref{eq:expotailtau} will show that $\sup_{p \geq 1} \mathbb{E}_{p}[\log^2( \overline{P}_{\theta_{1}}/p) ]$ is bounded. This will finish the proof of the lemma. To this aim we look at the tail \[ \mathbb{P}_{p}(\overline{P}_{k \lfloor p^{a-1} \rfloor} > x p)\] for $x >0$ large. We first reduce the problem from $\overline{P}$ to $P$ by a classical maximal inequality: We suppose that $x >k^{1/(a-1)}$ and we claim that there is a universal constant $c>0$ (independent of $k\geq 1$, $x> k^{1/(a-1)}$ and $p$) such that we have 
 \begin{eqnarray} \mathbb{P}_{p}( \overline{P}_{k\lfloor p^{a-1} \rfloor}>2xp) \leq c \cdot \mathbb{P}_{p}( P_{k\lfloor p^{a-1} \rfloor} > xp). \label{eq:boundsupbynormal}  \end{eqnarray}
The reason is that if the process $P$ reaches a value larger than $xp$ before time $kp^{a-1}$ then afterwards it has a positive probability to stay within $(k\lfloor p^{a-1} \rfloor)^{1/(a-1)} \leq xp$ of this value until time $k\lfloor p^{a-1} \rfloor$. We then use the relation with the non-conditioned random walk $(W)$ to evaluate the tail  of $P_{k\lfloor p^{a-1} \rfloor}$:
 \begin{eqnarray}  \mathbb{P}_{p}( P_{k\lfloor p^{a-1} \rfloor} > xp) &=& \sum_{y > xp} \mathbb{P}_{p}( W_{k\lfloor p^{a-1} \rfloor} = y \mbox{ and } W_{i} \geq 1,\forall 0 \leq i \leq k\lfloor p^{a-1}\rfloor) \frac{h^\uparrow(y)}{h^\uparrow(p)}\nonumber \\
 &\leq & \sum_{y > xp} \mathbb{P}_{p}( W_{k\lfloor p^{a-1} \rfloor} = y) \frac{h^\uparrow(y)}{h^\uparrow(p)}.  \label{eq:tailinter}\end{eqnarray}
A well-known ``one-jump'' principle (see e.g.~\cite{DDS08}) tells us that when $y$ is large, the main contribution to $\mathbb{P}_{p}( W_{k\lfloor p^{a-1} \rfloor} = y)$ is given by those events where the walk $W$ has one increment of size approximately $y$. In our case, there exists a constant $C>0$  which may vary from line to line such that 
 \begin{eqnarray*} \mathbb{P}_{p}(W_{k\lfloor p^{a-1} \rfloor} = y) =\mathbb{P}_{0}(W_{k\lfloor p^{a-1} \rfloor} = y-p) &\leq& C \cdot k\lfloor p^{a-1} \rfloor \cdot \mathbb{P}( \Delta W = y-p)\\  &\leq& C kp^{a-1} y^{-a}. \end{eqnarray*} Plugging this into \eqref{eq:tailinter} and using the fact that $h^\uparrow(\ell)$ grows like $\sqrt{\ell}$  as $\ell \to \infty$ it follows that
 \begin{eqnarray*}  \mathbb{P}_{p}( P_{k\lfloor p^{a-1} \rfloor} > xp)  &\leq& C\cdot k x^{-a+3/2}.  \end{eqnarray*} Using the above estimate together with \eqref{eq:boundsupbynormal} an easy calculation yields the estimate \eqref{eq:rough}. \endproof

\begin{lemma} \label{lem:logbehavior} We have the following two almost sure convergences
\[ \frac{\log P_{n}}{\log n} \xrightarrow[n\to\infty]{a.s.} \frac{1}{a-1},\]
\[ \frac{\log V_{n}}{\log n} \xrightarrow[n\to\infty]{a.s.}  \frac{a-1/2}{a-1}.\]
\end{lemma}
\proof The estimates of the first point of the lemma could be proved by bare hand calculations as those presented in the last lemma, however we chose a different and perhaps lighter route using Tanaka's construction of the walk $W^\uparrow$ conditioned  to stay positive \cite{Tan89}. To start with, let $ \mathrm{Exc}$ be the time and space reversal of a negative excursion of $W$:
\[  \mathrm{Exc} = ( 0, W_{\sigma}-W_{\sigma-1}, W_{\sigma}-W_{\sigma-2}, \ldots, W_{\sigma}-W_{1}, W_{\sigma})\] where $\sigma = \inf \{ k \geq 0 : W_{k} >0\}$. One then considers independent copies $ \mathrm{Exc}_{1}, \mathrm{Exc}_{2}, \ldots$ of $ \mathrm{Exc}$ which we concatenate together to get an infinite walk. Tanaka \cite{Tan89} proved that the process obtained has the law of $W^\uparrow$ (but started from $0$ and conditioned not to touch $ \mathbb{Z}_{<0}$). 

\begin{figure}[!h]
 \begin{center}
 \includegraphics[width=14cm]{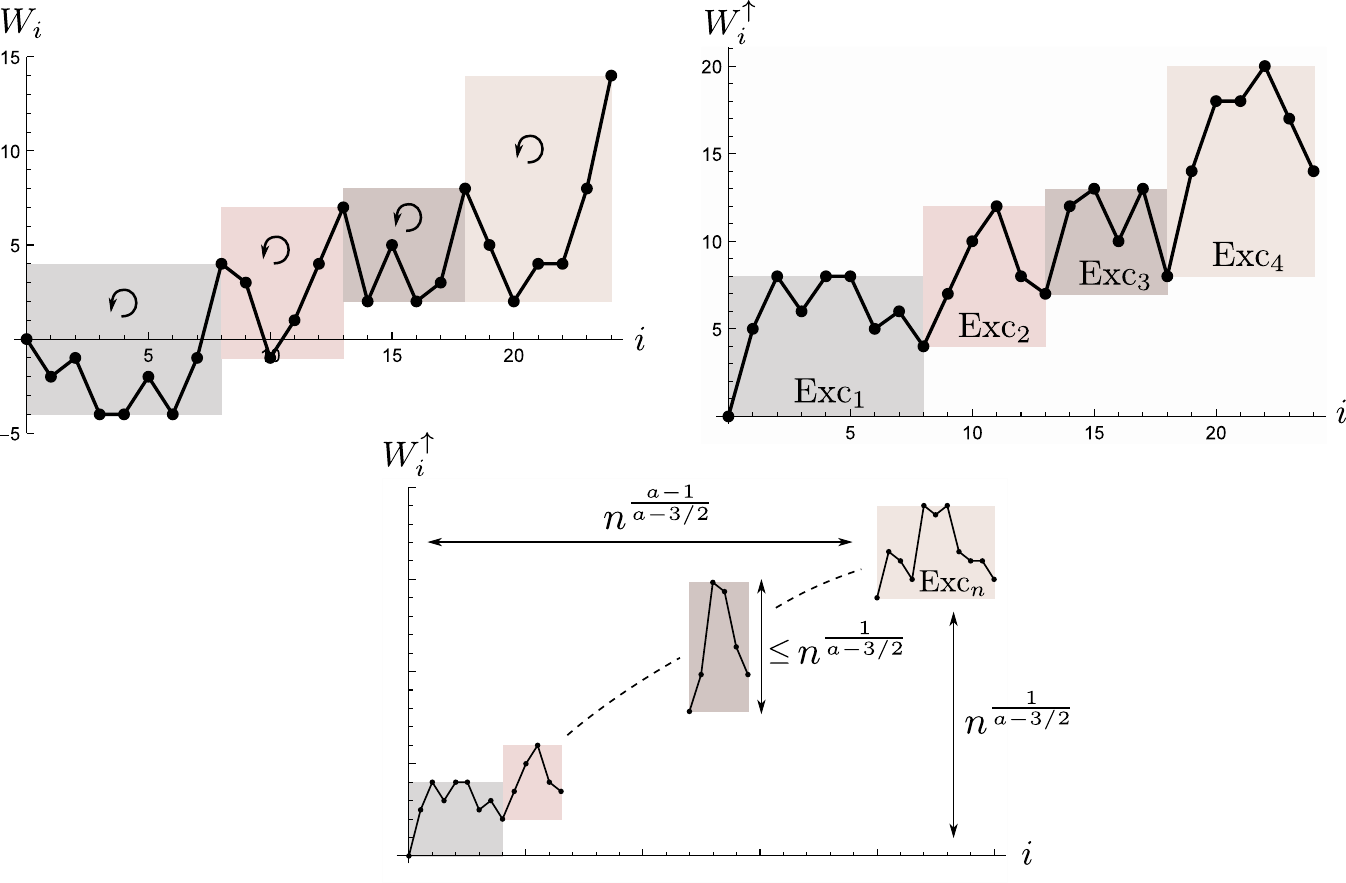}
 \caption{ \label{fig:tanaka} Illustration of Tanaka's construction of the walk $W^\uparrow$.}
 \end{center}
 \end{figure} We recall the following known tail estimates 
 \begin{eqnarray} \mathbb{P}( W_{\sigma} > x) &\sim& c_{1} \cdot x^{- (a-3/2)} \nonumber \\
 \mathbb{P}( \sigma > x) &\sim& c_{2} \cdot x^{- \frac{a-3/2}{a-1}}\nonumber \\
 \mathbb{P}( \max \mathrm{Exc} > x) &\leq& c_{3} \cdot x^{-(a-3/2)} \label{eq:estimatese}, \end{eqnarray}
 as $x \to \infty$ for some constants $c_{1},c_{2},c_{3}>0$. The first two estimates can be found in  \cite[Remark 1.2, Lemma 2.1]{CC08}  and the last one can be deduced from the second one: Indeed, for $x >0$ consider $\tau_{-x} = \inf \{i \geq 0 : W_{i} \leq -x\}$, then conditionally on the event $\tau_{-x}< \sigma$, the probability of the event \[\{|W_{k+\tau_{-x}} - W_{\tau_{-x}}| < x/2 : \forall 0 \leq k \leq x^{a-1}\}\] is bounded away from zero by some constant $c>0$ uniformly in $x >0$ (this follows from the Markov property and the convergence of $ x^{-1}W( \cdot\, x^{a-1})$ towards the $(a-1)$-stable Lévy process). In particular, on this event we have $\sigma > \tau_{-x} + x^{a-1}$ and therefore
 \[ \mathbb{P}( \sigma > x^{a-1}) \geq c \cdot \mathbb{P}(\min_{0 \leq i < \sigma} W_{i}\leq -x).\]
 But clearly we have $ \max  \mathrm{Exc}  \leq W_{\sigma}-\min_{0 \leq i < \sigma} W_{i}$ and so    \begin{eqnarray*}  \mathbb{P}(\max  \mathrm{Exc} > 2x) &\leq&  \mathbb{P}(\min_{0 \leq i < \sigma} W_{i}\leq -x) + \mathbb{P}( W_{\sigma}>x)\\
 & \leq & \frac{1}{c}  \mathbb{P}( \sigma > x^{a-1}) + \mathbb{P}( W_{\sigma}>x)\\
 & \underset{ \mathrm{asympt.}}{\leq} & \frac{1}{c} c_{2}  (x^{a-1})^{ \frac{-(a-3/2)}{a-1}} + c_{1} x^{-(a-3/2)},  \end{eqnarray*}
 and the desired third estimate of \eqref{eq:estimatese} follows. 
 
 We then use \eqref{eq:estimatese} in conjunction with the following classical result: if $S_{n} = X_{1}+ X_{2}+\cdots +X_{n}$ is a random walk whose increments are independent, non-negative and satisfy $ \mathbb{P}(X_{i}>x) \sim c\cdot x^{-1/\alpha}$ for $\alpha >1$ and $c>0$ (resp.\ $ \mathbb{P}(X_{i}>x) \leq c\cdot x^{-1/\alpha}$)  then we have   \begin{eqnarray} \label{eq:loglog} \frac{\log S_{n}}{\log n}  \xrightarrow[n\to\infty]{a.s.} \alpha  \qquad  ( \mbox{ resp.} \quad \ \limsup_{n \to \infty} \frac{\log S_{n}}{\log n} \leq \alpha).  \end{eqnarray} When applied to the above construction, this remark shows that after concatenating $n$ excursions, the total length is of order $n^{ \frac{a-1}{a-3/2}+o(1)}$, the current height  is of order $n^{1/(a-3/2)+o(1)}$ and the height of the largest excursion is no more than $n^{1/(a-3/2)+o(1)}$. Having a look at Fig.\,\ref{fig:tanaka} this implies that $W^{\uparrow}_{n} = n^{1/(a-1) + o(1)}$ as desired in the first point of the proposition.

Let us now turn our attention to the volume process. Recall that conditionally on the perimeter process $(P_{n})_{n\geq0}$ the volume process is obtained by summing the volume of Boltzmann maps each time the perimeter produces a negative jump. Let us bound the tail of $\Delta V_{n}$: for $x >0$ we have 
  \begin{eqnarray*} 
  \mathbb{P}(\Delta V_{n}>x) &=&  \sum_{\ell=1}^{\infty}\mathbb{P}( \Delta V_{n}> x \mbox{ and } \Delta P_{n} = -\ell  ) \\
   & = & \sum_{\ell=1}^{\infty} \mathbb{P}(|B^{(\ell-1)}| >x)\, \mathbb{P}(\Delta P_{n} = -\ell)\\
  & \underset{  \mathrm{Markov}}{\leq} &  \sum_{\ell=1}^{\infty} \left(x^{-1}\mathbb{E}[|B^{(\ell-1)}|] \wedge 1\right) \mathbb{P}(\Delta P_{n} = -\ell) \\
  & \underset{  \eqref{eq:bounddeltaperi} \mathrm{ \ and \ Prop.} \ref{prop:volume}}{\leq} & c \sum_{\ell=1}^{\infty}\left( \frac{\ell^{a- \frac{1}{2}}}{x} \wedge 1 \right) \ell^{-a} \\ 
  & \leq & c \, x^{- \frac{a-1}{a-1/2}},  \end{eqnarray*} for some constant $c>0$ that may vary from line to line. Using the uniform control over the tail of $\Delta V_{n}$ we can stochastically bound from above the volume process $(V_{n})_{n \geq 0}$ by a process $(\tilde{V}_{n})_{ n\geq 0}$ with independent positive increments with a tail of order $ \mathbb{P}( \Delta \tilde{V}_{n}>x) \sim c  x^{- \frac{a-1}{a-1/2}}$. And so by \eqref{eq:loglog} we deduce that \[ \limsup_{n \to \infty } \frac{\log {V}_{n}}{\log n }\leq \limsup_{n \to \infty }\frac{\log  \tilde{V}_{n} }{\log n } \underset{ \eqref{eq:loglog}}{\leq} \frac{a-1/2}{a-1}.\] For the lower bound we use the fact that $V_{n}$ dominates any of its jump until time $n$. Since the process $P_{n}$ makes negative jumps of order $n^{1/(a-1)}$ until time $n$, the process $(V)$ makes jumps of order $n^{(a-1/2)/(a-1)}$ until time $n$. We leave it to the reader to turn this heuristic into an almost sure lower bound.
\endproof 

\section{A special weight sequence}
\label{sec:particular}

In this paper we have considered general weight sequences $\qseq$ with asymptotic behaviour $q_k \sim c \kappa^{k-1}k^{-a}$.
Let us wrap up by revisiting some of the results for a very convenient particular weight sequence \cite{ABM16} for $a\in(3/2;5/2)$ given by
\begin{equation}\label{eq:specweights}
q_k = c \kappa^{k-1} \frac{\Gamma(\half-a+k)}{\Gamma(\half+k)}\mathbf{1}_{k\geq 2}, \quad\quad \kappa = \frac{1}{4a-2},\quad\quad c = \frac{-\sqrt{\pi}}{2\,\Gamma(3/2-a)}.
\end{equation}
Notice that this weight sequence is term-wise continuous as $a\to5/2$ taking the value $q_k = \frac{1}{12}\mathbf{1}_{k=2}$, which corresponds exactly to critical quadrangulations.

\begin{lemma}
For $a\in(3/2,5/2)$ the weight sequence (\ref{eq:specweights}) is admissible and critical and the law $\nu$ of the corresponding random walk $(W_i)_{i}$ is given by
\begin{equation}\label{eq:nuspec}
\nu(k) = c\frac{\Gamma(3/2-a+k)}{\Gamma(3/2+k)}\mathbf{1}_{k\neq 0}.\quad\quad (k\in\Z)
\end{equation}
\end{lemma}
\begin{proof}
Clearly the values $\nu(k)$, $k\in\Z$, are nonnegative and one may check that the characteristic function $\phi(\theta) := \sum_{k=-\infty}^\infty \nu(k) e^{ik\theta}$ of (\ref{eq:nuspec}) is given by
\begin{equation*}
\phi(\theta) = 1 - \frac{\pi}{2}\frac{\Gamma(a-1/2)}{\Gamma(a)} (1-e^{i\theta})^{a-3/2}\sqrt{1-e^{-i\theta}}. 
\end{equation*}
Since $\phi(0)=1$, $\nu$ defines a probability measure on $\Z$.
Using that $\kappa = \nu(-1)/2$, it follows from Proposition \ref{prop:qnubijection} that the only thing we need to check is that $h^\uparrow$ is $\nu$-harmonic on $\Z_{>0}$, i.e., that
\begin{equation}\label{eq:nuharmonic}
\sum_{k=-\infty}^\infty h^\uparrow(\ell+k)\nu(k) = h^\uparrow(\ell) \quad\text{for }\ell>0.
\end{equation}
Using that $\sum_{\ell=1}^\infty h^\uparrow(\ell)e^{-i\ell\theta} = e^{-i\theta}(1-e^{-i\theta})^{-3/2}$ we find for $\ell>0$ that
\begin{align*}
\sum_{k=-\infty}^\infty h^\uparrow(\ell+k)(\nu(k)-\mathbf{1}_{k=0}) &= \frac{1}{2\pi} \int_0^{2\pi} \frac{e^{(\ell-1)i\theta}}{(1-e^{-i\theta})^{3/2}} (\phi(\theta)-1)\rmd\theta\\
&= -\frac{\Gamma(a-1/2)}{4\Gamma(a)} \int_0^{2\pi} e^{i\ell\theta}(1-e^{i\theta})^{a-5/2}\rmd\theta = 0,
\end{align*}
which implies (\ref{eq:nuharmonic}).
\end{proof}

The scaling constants in Theorem \ref{thm:scalingperimeter+volume} take on the values
\begin{align*}
\mathsf{p}_\qseq = c^{\frac{1}{a-1}},\quad\quad \mathsf{b}_\qseq = \frac{1}{\Gamma(a+1/2)},\quad\quad \mathsf{v}_\qseq = \frac{1}{\Gamma(a+1/2)}c^{\frac{a-1/2}{a-1}}.
\end{align*}
On the other hand, for $a \in (2,5/2)$,
\begin{align*}
\mathsf{a}_\qseq &:= \frac{1}{2}\left(1+ \sum_{k=0}^\infty (2k+1)\nu(k)\right) = 1 + \frac{1}{4(a-2)},\quad\quad h_\qseq = \mathsf{a}_\qseq/(2p_\qseq)
\end{align*}
and for $a\in(3/2,2)$,
\begin{align*}
\mathbb{E}\left[\mathrm{d_{fpp}}( \rootface, \infty)\right] = \sum_{k=1}^\infty \prob_1(W_k=0) = \frac{1}{2\pi}\int_0^{2\pi} \frac{e^{i\theta}\rmd \theta}{1-\phi(\theta)} = \frac{\cot(\pi a)}{\pi} \frac{a-1}{(a-\frac{5}{2})(a-\frac{3}{2})}.
\end{align*}

\end{document}